\documentclass[12pt,leqno]{amsart}
\usepackage{amsmath, amstext, amssymb, amsthm, a4wide,enumerate}
\usepackage{epic,eepic,epsf,graphicx,color}
\usepackage[english]{babel}
\newtheorem{theorem}{Theorem}
\newtheorem{proposition}[theorem]{Proposition}
\newtheorem{lemma}[theorem]{Lemma}

\newtheorem{definition}[theorem]{Definition}

\title{A Compressible Multifluid System  \\  \hskip-.8cm  with New Physical Relaxation Terms}
\date{\today}
\author{D. Bresch, M. Hillairet}

\begin{document}

\maketitle

\begin{abstract}
   In this paper, we rigorously  derive a new compressible multifluid system from compressible Navier-Stokes equations with density-dependent viscosity in the one-dimensional in space setting.  More precisely, we propose and mathematically derive a generalization  of the usual  
one velocity Baer-Nunziato model  with a new relaxation term in the PDE governing the volume fractions. This new
relaxation term encodes the change of viscosity and pressure between the different fluids.
  For the  reader's convenience,  we first establish a  formal derivation in the bifluid setting using a WKB decomposition and then we  rigorously justify the multifluid homogenized system 
using a kinetic formulation via Young measures characterization.
\medskip


\medskip


\end{abstract}

\section{Introduction}
    This article is devoted to the mathematical derivation of multifluid systems with one velocity and  with variable viscosity. 
    This generalizes the usual Baer-Nunziato system with one velocity, already justified in \cite{BrHi}, by modifying the PDE on the fractions through the relaxation term. 
   This term takes into account  the change of viscosity and pressure between the different fluids.
       
      If we look at physical books such as those written by M. {\sc Ishii} and T. {\sc Hibiki}  (see \cite{IsHi}) or by D.~{\sc Drew} and S.L. {\sc Passman} (see \cite{DrPa}), we understand well that it is not so easy to choose the averaging process that has to be used to derive appropriate multifluid systems and that formal closure assumptions are  all the times made in physicist's articles.
    How to derive appropriate multifluid systems reflecting interface flux laws? How to recover mathematically a PDE governing  the fraction of each component with physical relaxation terms? These issues require a careful study of the interface  evolution between the phases. Unifying the equations for all the phases into a single compressible Navier Stokes equation allows to follow the dynamics of these interfaces
(with an appropriate notion of solution  enabling to control the divergence of the velocity). Multifluid systems are then interpreted as reduced systems satisfied by particular Young measure (namely convex combinations of a finite number of Dirac masses) solutions to the homogenized  compressible Navier Stokes equation.  Proving propagation of the number of Dirac masses in Young measure solutions to this homogenized equation is then  the key-point to derive the multifluid system with new relaxation terms. 
  In this paper, we decide to work on the compressible Navier--Stokes equations with density-dependent viscosity in the one-dimensional in space setting. The density-dependent framework enables to consider phases with different viscosities. We first prove that it is possible to define an appropriate sequence of  weak solutions on which we can perform the homogenization process.  Our result generalizes to compressible Navier--Stokes equations with density-dependent viscosity the work performed in the one dimension in space by \cite{Se} related to compressible Navier--Stokes equations  with constant viscosity.

\smallskip

    Let describe in more details the method we apply here. Consider for instance in a three-dimensional container $\Omega$,  the mixture of two viscous compressible phases described by  triplets density/velocity/pressure $(\rho_+,u_+,p_+)$ and $(\rho_-,u_-,p_-)$ respectively.
    Introducing $(\mu_{\pm},\lambda_{\pm})$ and $\mathcal P_{\pm}$ the respective viscosities and pressure laws of the phases,  we obtain that, for $i = +,-$ the triplet is a solution to the compressible Navier Stokes equations
 \begin{eqnarray*}
 \partial_t \rho_{i} + {\rm div} (\rho_i u_i) &=& 0 \\
 \partial_t (\rho_i u_i) + {\rm div} (\rho_i u_i\otimes u_i) &=& {\rm div}\Sigma_i  
 \end{eqnarray*}  
 on its domain $\mathcal F_{i}(t),$  with the equation of state:
 \begin{eqnarray*}
  \Sigma_i &=& \mu_i (\nabla u_i + \nabla^{\top} u_i) + (\lambda_i {\rm div} u_i - p_i) \mathbb I_3 \\
 p_i &=& \mathcal P_i (\rho_i).
\end{eqnarray*}
 Neglecting the properties of the interfaces, so that: 
\begin{itemize}
\item $ \mathcal F_{+} \cup \mathcal F_{-} \cup (\overline{\mathcal F_+} \cap \overline{\mathcal F_-})= \Omega \,,$      
\item the phases do not slip one on the other at the interface,
\item we have continuity of the normal stresses at the interface,
\end{itemize}
we have that the extended unknowns 
 $$
\rho = \rho_+ \mathbf{1}_{\mathcal F_+} + \rho_- \mathbf{1}_{\mathcal F_-} \quad u =  u_+ \mathbf{1}_{\mathcal F_+} + u_- \mathbf{1}_{\mathcal F_-} 
 $$
satisfy the compressible Navier Stokes equations on the whole container $\Omega:$
 \begin{eqnarray} \label{eq_CNS1}
 \partial_t \rho + {\rm div} (\rho u) &=& 0 \\
 \label{eq_CNS2}\partial_t (\rho u ) + {\rm div} (\rho u \otimes u) &=& {\rm div}\Sigma  .
 \end{eqnarray}  
 Assuming further that the densities of the different phases range two non-overlapping intervals $I_+$ and $I_-$, we can complement this system by the equations of states:
  \begin{eqnarray}
\label{eq_CNS3}  \Sigma_i &=& 2 m (\nabla u + \nabla^{\top} u) + (l {\rm div} u - p) \mathbb I_3 \\
\label{eq_CNS4}  p &=& \mathcal P (\rho)  \quad m = \mathcal M(\rho) \quad l = \Lambda(\rho)\,, 
\end{eqnarray}
with functions $\mathcal P, \mathcal M,\Lambda$ such that for $i\in \{+,-\}$ we have:
$$
\mathcal P(\rho) = p_i(\rho), \quad  \mathcal  M(\rho) = m_i ,\quad \Lambda(\rho)  = \lambda_i, \quad \forall \, \rho \in I_i.
$$
We aim here to compute solutions to this system where any time/space cell of arbitrary small size contains
a fraction of phase $+$ and a fraction of phase $-.$ Thus, our problem reduces to a mathematical study of the 
homogenization of solutions to the extended compressible Navier Stokes equation \eqref{eq_CNS1}--\eqref{eq_CNS4} with respect to initial density. 
This  method for the justification of multifluid systems has been successfully applied in the multi-dimensional  setting recently by the authors starting from the compressible Navier--Stokes equation with constant viscosity  \cite{BrHi}. The interested reader is also referred to \cite[Section 7]{PlSo} where the kinetic equation formulation has been proposed in terms of the cumulative distribution function and without characterization of the Young measures which gives multifluid systems. One corollary of the results in this paper is that the method is robust as it extends to the viscosity-dependent case making precise, in the multifluid setting, previous results initiated by {\sc A.A. Amosov} and {\sc A.A. Zlotnik} in the 90's, see \cite{AmZl} and references cited therein. Note also that we do not use the Lagrangian formulation but directly work on the Eulerian system.

   In the one-space dimension setting, that we consider in this article,  we construct appropriate solutions $(\rho,u,p)$ of the following system
\begin{eqnarray} 
\label{eq_CNS1D1} \partial_t \rho + \partial_x(\rho u) &=& 0 \,, \\
 \label{eq_CNS1D2} \partial_t \rho u + \partial_x (\rho u^2) &=& \partial_x (\mu(\rho) \partial_x  u) - \partial_x p \,, \\ 
p &=& p(\rho)  \label{eq_CNS1D3}
\end{eqnarray}
and derive a kinetic equation which governs the evolution of the homogenized system (in terms of Young measures/velocity).
With Young measures written as the convex combination of $k$ Dirac masses, we obtain then the multifluid system that  reads:
\begin{align}
& \partial_t \alpha_i + \partial_x (\alpha_i u)   = \dfrac{\alpha_i}{\mu(\rho_i)}f_i \\
& \partial_t \rho_i  + u \partial_x \rho_i =  - \dfrac{\rho_i}{\mu(\rho_i)}f_i \\ 
& \partial_t (\rho u) + \partial_x (\rho u^2)  = \partial_x [ \mu \partial_x u - p] 
\end{align}
for $1\le i\le k$ with 
\begin{eqnarray*}
f_i &=& \dfrac{1}{ \displaystyle \left[\sum_{j=1}^{k} \dfrac{\alpha_j}{\mu(\rho_j)}\right]} \left(\partial_x u - \sum_{j=1}^k \alpha_j \dfrac{p(\rho_j)}{\mu(\rho_j)} \right) +  p(\rho_i)\quad \text{ for $1 \leq i \leq k$} \\
\displaystyle \mu &=& \frac{1}{\displaystyle \sum_{i=1}^k 
\displaystyle \frac{\alpha_i}{\mu (\rho_i)}}, \qquad
 p =  \frac{ \displaystyle  \sum_{i=1}^k \frac{\alpha_i p(\rho_i)}{\mu(\rho_i)}}{\displaystyle
 \sum_{i=1}^k  \displaystyle \frac{\alpha_i}{\mu(\rho_i)}}.
\end{eqnarray*}
Note that in the bifluid setting,  to get such a system we can also formally search  for two-scale solutions (a kind of WKB expansion) under the following form:
\begin{eqnarray} \label{eq_expansion1}
\rho(t,x) &=& \sum_{i=+,-} \theta_i\left(t, \dfrac t \varepsilon , x , \dfrac x \varepsilon \right) \rho^{\varepsilon}_i(t,x)\,, \\ u(t,x) & = & u_0\left(t, \dfrac t \varepsilon , x , \dfrac x \varepsilon \right) + \varepsilon u_1  \left(t, \dfrac t \varepsilon , x , \dfrac x \varepsilon \right) + \varepsilon^2 u_2  \left(t, \dfrac t \varepsilon , x , \dfrac x \varepsilon \right) + O(\varepsilon^3)\,, \label{eq_expansion2}
\end{eqnarray}
assuming that 
\begin{equation} \label{eq_expansionHypo}
\rho^{\varepsilon}_i(t,x) = \rho_{i}^0(t,x) + O(\varepsilon) \,, \quad \theta_i(t,\tau,x,y) \in \{0,1\} \quad \text{a.e.}.
\end{equation}
After some calculations, the general system that we obtain on $(\alpha_\pm,  u_0,  \rho^0_\pm, \overline p)$, where $\alpha_\pm$ denotes the average with respect to the fast variables $(\tau,y)$ of $\theta_\pm$,  reads
\begin{eqnarray}
&&  \alpha_+ + \alpha_- = 1, \\[8pt]  
&&  \partial_t  \alpha_+  + u_0\partial_x  \alpha_+ =  \frac{\alpha_+ \, \alpha_-}{
        \alpha_- \mu_+^0 +  \alpha_+ \mu_-^0} [(p_+^0-p_-^0) 
         + (\mu_-^0 -\mu_+^0) \partial_x u_0]\\[8pt]
&& \partial_t(\alpha_+ \rho_+^0) + \partial_x( \alpha_+ \rho_+^0 u_0)
=  0 \,, \\[4pt]
 &&\overline\rho (\partial_t u_0 + u_0 \partial_x u_0)
     - \partial_x(m  \partial_x u_0) + \partial_x  \pi 
     = 0\,, \\[8pt]
\label{rel2}&& \mu_+^0 = \mu(\rho_+^0), \qquad \mu_-^0  =   \mu(\rho_-^0), \qquad
    m = \dfrac{\mu_+^0 \mu _-^0}{\alpha_+ \mu_-^0 + \alpha_- \mu_+^0}\,, \\
    \label{rel1} && p_+^0 = p(\rho_+^0), \qquad p_-^0  =   p(\rho_-^0), \qquad
    \overline \rho = \alpha_+ \rho_+^0 +  \alpha_-\rho_-^0,  \qquad
    \pi = \dfrac{\alpha_+ p_+^0 \mu_-^0  + \alpha_- p_-^0 \mu_+^0}{\alpha_+ \mu_-^0 + \alpha_- \mu_+^0}. 
\end{eqnarray}
Remark that, in the two-fluid setting, we obtain a new equation on the fraction $\alpha_+$ in which the difference on the effective fluxes plays a crucial role. 
Namely, we obtain
$$ \partial_t \alpha_+ + u_0\partial_x \alpha_+
     = \frac{\alpha_+ \, \alpha_-}{
     \   \alpha_-  \mu_+^0 +  \alpha_+\mu_-^0} [F_+-F_-].
$$
where $F_\pm = -\mu_\pm^0 \partial_x u_0 + p_\pm^0$ and $p^0_\pm$ and $\mu_\pm^0$
are defined by \eqref{rel1}--\eqref{rel2}. 
In the particular case $\mu(\rho) = \mu = \hbox{cste },$ the system reduces to the one-dimensional system that has been formally derived by W.E. in \cite{E} and  fully mathematically  justified by D.~{\sc Serre} in \cite{Se}.  In that case, the PDE on $\alpha_+$ simplifies as
$$ \partial_t  \alpha_+ + u_0\partial_x \alpha_+
     = \frac{\alpha_+ \, \alpha_-}{
      \mu} (p_+^0-p_-^0).
$$
In the appendix B, we show that the system obtained through the formal WKB method and the system derived using kinetic formulation and characterization of the Young measures are the same.
 
   The outline of the paper is as follows. We start by our mathematical results. Then,  we provide a formal derivation for  two-fluid flows plugging the WKB ansatz mentioned above in the compressible Navier--Stokes system with density-dependent viscosity and identifying the different terms. Finally, we justify this formal calculation to get a multifluid system with  $k$ phases using homogenization technics through Young-measures characterization.  The control of the divergence of the velocity field is a key-point of our analysis as it enables to follow the dynamics of the interfaces.  In an appendix, we recall some well-known results on the transport equation and as mentioned previously we compare the result obtained formally and the one obtained through Young-measures characterization in the bifluid setting.


\section{Mathematical results.}
  In this section we make precise the assumptions on the equations of state in the system under consideration. Then,  we give the first result of existence which will be used in the homogenization process.

      We consider the following Navier--Stokes system with density-dependent viscosity:
\begin{equation} \label{eq_NSC}
\left\{
\begin{array}{rcl}
\partial_t \rho + \partial_x (\rho u) &=& 0 \,, \\
\partial_t \rho u + \partial_x (\rho u^2) &=& \partial_x [ \mu \partial_x  u] - \partial_x p \,, \\
\end{array}
\right.
\text{ on $(0,L)$}
\end{equation}
completed with the equations of state: 
\begin{equation}  \label{eq_ee}
p = p(\rho) \qquad \mu = \mu(\rho) 
\end{equation}
where $p$ and $\mu$ are given and sufficiently smooth: we assume throughout the paper that
\begin{align}
& p \in C^1([0,\infty)) \,, \quad \text{ with }  \quad p'(s) \geq  0, & \forall \, s \in [0,\infty). \\
& \mu \in C^1([0,\infty)) \,, \quad \text{ with } \quad \mu(s) \geq \mu^0(1+\sqrt{s}), & \forall \, s \in [0,\infty).  \label{eq_assummu}
\end{align} 
where  $\mu^0$ is a given strictly positive constant. The assumption on $p$ can be relaxed into:
\begin{align} \label{eq_p2}
& p \in C^1([0,\infty)) \,, \quad \text{ with }  \quad p'(s) \geq 0\,,  & \forall \, s >>1.
\end{align}
But we compute with the previous one for simplicity.
 We complement the above pdes with:
\begin{itemize}
\item periodic boundary conditions in $x$
\item initial conditions:
\begin{equation} \label{eq_id}
\rho(0,x) = \rho^0(x) \,, \quad u(0,x) = u^0(x) \,.  
\end{equation}
\end{itemize}

\bigskip

\paragraph{\bf Conventions for periodic functions.} We denote indifferently with $\sharp$ or quotient-set periodic-function 
spaces. For instance $L^{\infty}(\mathbb R/L\mathbb Z) = L^{\infty}_{\sharp}.$ The symbol
$L$ is implicit for any sharped notations. 
 For a Banach space $X$ such as $L^p$ or $H^m,$ we endow $X_{\sharp}$ with the norm:
$$
\|u\|_{X_{\sharp}} = \|u_{|_{(0,L)}}\|_{X(0,L)}.
$$ 
We recall that $X_{\sharp}$ is a Banach space endowed with this norm and that a sequence $u_n$ converges toward $u$ in $X_{\sharp}$ for the strong topology (resp. for the weak or the weak$-*$ topology) if and only if $u_n$ converges toward $u$ in $X(-M,M)$ for
the strong topology (resp. for the weak or the weak$-*$ topology), whatever the value of $M \in L\mathbb N^*.$ In particular, if $u_n$ converges toward $u$ in $X$ (endowed with the weak/weak$-*$ topology)  then $u_n$ converges toward $u$ in $\mathcal D'(\mathbb R).$

\bigskip

Our first target result  is the following theorem:
\begin{theorem} \label{thm_existence1}
Given $\rho^0 \in L^{\infty}_{\sharp}$ and $u^0 \in H^1_{\sharp}$ satisfying
\begin{equation} \label{eq_hyporho0}
\underline{\rho^0} := \inf \rho^0(x) >0, \qquad \overline{\rho^0} = \sup \rho^0(x) < \infty\,, 
\end{equation}
there exists $T_0$ depending on $\underline{\rho^0},\overline{\rho^0},\|u^0\|_{H^1_{\sharp}}$
such that there exists at least one pair $(\rho,u)$ for which:
\begin{itemize}
\item[$({\mathbf{HDS}})_a$] we have the regularity statements:
\begin{align}
 & \rho \in L^{\infty}((0,T_0); L^{\infty}_{\sharp}) \cap C([0,T_0] ; L^1_{\sharp}) \,,  \\
 & u \in L^{\infty}((0,T_0);H^{1}_{\sharp})  \cap C([0,T_0];L^2_{\sharp})    \,,      \\
 &  z := \mu(\rho) \partial_x u - p(\rho) \in L^{2}((0,T_0);H^1_{\sharp}) ;
 \end{align}
\item[$({\mathbf{HDS}})_b$] $(\rho,u)$ satisfies \eqref{eq_NSC} in $\mathcal D'((0,T) \times \mathbb R),$ with $p,\mu$ given 
by \eqref{eq_ee},  and matches initial conditions \eqref{eq_id} in $L^2_{\sharp} \times H^1_{\sharp}\,,$
\item[$({\mathbf{HDS}})_c$] we have the following bounds :
\begin{enumerate}[$\bullet$]
\item for a.e. $(t,x) \in (0,T_0) \times (\mathbb R/L\mathbb Z) $ there holds:
\begin{align} \label{eq_regathm}
\dfrac{1}2 \underline{\rho^0} \leq \rho(t,x) \leq 2 \overline{\rho^0} \,,
\end{align}
\item for a.e. $t \in (0,T_0)$ there holds (see \eqref{eq_defq} for the definition of $q$):
\begin{align} \label{eq_regbthm}
& \int_{0}^L \left[ \dfrac 12|u(t,\cdot)|^2 +  q(\rho(t,\cdot)) \right]    + \int_0^t \int_0^L \mu |\partial_x u|^2   \leq  \int_{0}^L \left[ \dfrac 12|u^0|^2 +  q(\rho^0) \right]  
\end{align}
\item there exists a constant $K_0$ depending only on $\underline{\rho^0},\overline{\rho^0}$ and $\|u^0\|_{H^1_{\sharp}}$ for which
\begin{align} \label{eq_regcthm}
& \sup_{t \in (0,T_0)} \|u(t,\cdot)\|_{H^1_{\sharp}}^2 + \int_0^{T_{0}} \|\partial_x z(t,\cdot)\|^2_{L^2_{\sharp}} \leq K_0.
\end{align}
\end{enumerate}
\end{itemize}
 \end{theorem}

We call solutions in the sense of {\bf Theorem \ref{thm_existence1}} HD  solutions to \eqref{eq_NSC}-\eqref{eq_ee} (after D.  {\sc  Hoff} and B. {\sc Desjardins} who constructed independently
such solutions for the constant viscosity case).
The scheme of our proof follows classical lines but we write details for reader's convenience :
\begin{itemize}
\item first, we obtain classical solutions to a regularized version of our system using the BD entropy procedure. This procedure in one-D is well known since the work in 1968 by  Y. {\sc Kanel} in \cite{Ka}.
  In our case, "regularized" only means that we assume the initial data to satisfy further $\rho^0 \in H^1_{\sharp}$ ;
\item second, we prove that the strong solutions are HD-solutions on some time-interval $(0,T_0)$ where $T_0$ depends only on $\underline{\rho^0},\overline{\rho^0},\|u^0\|_{H^1_{\sharp}}$ ;
\item third, we apply a compactness argument showing that a sequence of solutions to the regularized system converges to the solution whose existence is claimed  in our theorem.
These HD solutions provide us with the solutions on which we  justify the homogenized procedure through Young measures.
\end{itemize}

 Enlarging the range of the  compactness argument, we also obtain the main result of this paper,
 namely, the mathematical justification of a generalization of the Baer-Nunziato with one velocity.
 More precisely we obtain the following mathematical result
\begin{theorem} \label{thm_multi}
Let $T_0 >0$ and  $(\rho_n,u_n)_{n\in \mathbb N}$ be a sequence of solutions to \eqref{eq_NSC}-\eqref{eq_ee}  on $(0,T_0),$ in the sense of {\bf Theorem \ref{thm_existence1}}, 
with  respective initial data $\rho^0_n \in L^{\infty}_{\sharp}$ and $u^0_n \in H^1_{\sharp}.$ Assume that the sequence of initial data 
satisfies 
\begin{enumerate}[$\bullet$]
\item $u^0_n \rightharpoonup u^0$ in $H^1_{\sharp}-w$
\item there exists a constant $C_0 >0$ such that  $ 1/C_0 \leq \rho^0_n \leq C_0$ uniformly,
\item there exists $(\alpha^0_i,\rho^0_i)_{i=1,...,k} \in [L^{\infty}_{\sharp}]^{2k}$ such that $\rho^0_n$ converges in the sense of Young measure  (to be defined below) towards
$$
\nu^0 = \sum_{i=1}^{k} \alpha_i^0  \delta_{\xi=\rho_i^0}
$$
\end{enumerate}
Then, up to the extraction of a subsequence, $(\rho_n,u_n)$ converges to 
$((\alpha_i,\rho_i)_{i=1,\ldots,k},u)$ (in a sense to be made precise) for which we have:
\begin{enumerate}[$\bullet$]
\item the regularity statements:
\begin{align}
 & \alpha_i \in L^{\infty}((0,T_0); L^{\infty}_{\sharp}) \cap C([0,T_0] ; L^1_{\sharp}) \text{ with }   \\
 \notag  & \qquad \qquad \alpha_i \geq 0 \,, \qquad  \forall i \in \{1,\ldots,k\} \,, \qquad \, \sum_{i=1}^k \alpha_i = 1\,, \quad  \text{ a.e. } \\
 & \rho_i \in L^{\infty}((0,T_0); L^{\infty}_{\sharp}) \cap C([0,T_0] ; L^1_{\sharp}) \text{ with }    \\
\notag  & \qquad \qquad  C_0/2 \leq \rho_i \leq 2 C_0 \text{ a.e. } \\
 & u \in L^{\infty}((0,T_0);H^{1}_{\sharp})  \cap C([0,T_0];H^{1}_{\sharp}-w) ;         
 \end{align}
\item the partial differential system (in the sense of $\mathcal D'((0,T_0) \times \mathbb R)$):
\begin{align}
& \partial_t \alpha_i + \partial_x (\alpha_i u)   = \dfrac{\alpha_i}{\mu(\rho_i)}f_i, \\
& \partial_t \rho_i  + u \partial_x \rho_i =  - \dfrac{\rho_i}{\mu(\rho_i)}f_i, \\ 
& \partial_t (\rho u) + \partial_x (\rho u^2)  = \partial_x [ m \partial_x u - \pi ],
\end{align}
where :
$$
\rho = \sum_{j=1}^k \alpha_j \rho_j,  \qquad  m = \left[ \sum_{j=1}^k \dfrac{\alpha_j }{\mu(\rho_j)}\right]^{-1}, \qquad \pi =  m \sum_{j=1}^k \alpha_j \dfrac{p(\rho_j)}{\mu(\rho_j)}.
$$
and
$$
f_i = \dfrac{1}{ \displaystyle \left[\sum_{j=1}^{k} \dfrac{\alpha_j}{\mu(\rho_j)}\right]} \left(\partial_x u - \sum_{j=1}^k \alpha_j \dfrac{p(\rho_j)}{\mu(\rho_j)} \right) +  p(\rho_i).
$$
\item  the initial conditions:
\begin{align}
& \alpha_i(0,\cdot) = \alpha_i^0 \text{ in $L^1_{\sharp}$}, \\
&\rho_i(0,\cdot) = \rho_i^0  \text{ in $L^1_{\sharp}$}, \\
& u(0,\cdot) = u^0 \text{ in $H^1_{\sharp}$}.
\end{align}
\end{enumerate}  
\end{theorem}



\section{Formal derivation for bifluid flows.}
In this part, we prove how to get the bifluid system using a formal
WKB decomposition. The reader interested by some formal 
papers related to heat-conducting case or to non-monotone pressure discussions
are referred to \cite{E} and \cite{Se1}.

We assume throughout this section that $(\rho,u)$ is a solution to the compressible Navier Stokes
system \eqref{eq_CNS1D1}-\eqref{eq_CNS1D2}-\eqref{eq_CNS1D3} given by the expansion \eqref{eq_expansion1}-\eqref{eq_expansion2} 
in which  \eqref{eq_expansionHypo} is satisfied.

\subsection{General setting}
Let us formally multiply the continuity equation by a function $\beta'$ such that:
$$
\beta= 1 \text{ on the support of $\rho^{0}_+$}, \qquad
\beta= 0 \text{ on the support of $\rho^{0}_-$}.
$$
We get the classical equation
$$
\partial_t \beta(\rho) + \partial_{x} ( \beta(\rho) u ) + ( \rho \beta'(\rho) - \beta(\rho)) \partial_x u  = 0.
$$
Replacing $\beta(\rho)$ by its value, we get the supplementary equation
\begin{equation}
\partial_{t} \theta_{+} + u \partial_{x} \theta_+  = 0.
\end{equation}
Then, we can decompose the derivatives in terms of the slow variables  $(t,x)$ and fast variables
 $(\tau,y)$ of $\theta_+.$ 
 We get two equations when we consider terms which are $O(1/\varepsilon)$ or terms
 which are $O(1)$: 
\begin{eqnarray} \label{eq_alpha-1}
\partial_{\tau} \theta_+ + u_0 \partial_{y} \theta_+ &=& 0 \, \\
\partial_{t} \theta_+ + u \partial_x \theta_+ &=& -\dfrac{(u - u_0)}{\varepsilon} \partial_y \theta_+. 
\end{eqnarray}
 The first equation provides the behavior of $\theta_+$ on a cell. This equation is consistent with the assumption that $\theta_+$ is an indicator function. Averaging with respect to
 the fast variable the second equation, we get the following PDE on the averaged quantity  $\alpha_+ = \overline{\theta_+}$ 
\begin{equation} \label{eq_alpha0}
 \partial_{t} \alpha_+  + \overline{u \partial_x \theta_+} = - \overline{\dfrac{(u - u_0)}{\varepsilon} \partial_y \theta_+}. 
\end{equation}
We denote temporarily with bars averages on a cell. As this lightens notations a lot, we keep this convention throughout this section only. 
However, it must not be confused with lower and upper bounds for densities as it has been used in the statement of our {\bf Theorem \ref{thm_existence1}} and as will be in the next section. 
Remark that there is no vacuum in the mixture so that $\theta_+ + \theta_- = 1$ a.e.. Consequently, we have:
$$
\alpha_+ + \alpha_- = 1.
$$

Choosing then $\beta'$ such that:
$$
\beta= \rho \text{ on the support of $\rho^{0}_+$}\,, \qquad
\beta= 0 \text{ on the support of $\rho^{0}_-$}\,.
$$
We obtain that 
$$
\partial_t (\theta_+ \rho^0_+) + \partial_x( \theta_+ \rho^0_+ u) = 0
$$
Keeping only the first order in $u$ and averaging with respect to fast variables, we obtain then:
\begin{equation} \label{eq_alpha2}
\partial_t (\alpha_+ \rho^0_+) + \partial_x (\rho^0_+\overline{\theta_+ u_0} ) = 0
\end{equation}

Now the main objective is to calculate the averaged terms in \eqref{eq_alpha0}-\eqref{eq_alpha2} and obtain the momentum equation. To proceed, 
we use the other equation. We distinguish two cases: the constant viscosity case
which gives at the end the system which has been justified recently in \cite{Se} (generalized 
to the multi-dimensioncal case in  \cite{BrHi})  and the density-dependent viscosity case which 
gives at the end the homogenized system under consideration in this paper.

\subsection{The constant viscosity case}
Plugging the expansion of $u$ in  \eqref{eq_alpha0}, we obtain at first order in $\varepsilon$:
\begin{equation} \label{eq_alpha01}
 \partial_{t} \alpha_+ + \overline{u_0 \partial_x \theta_+} =  -\overline{u_1 \partial_y \theta_+}. 
\end{equation}
It remains then to compute the averaged quantity on the right-hand side and to justify the homogenized momentum equation.
Let us recall quickly the different steps to get the limit system which asks to interprete the divergent
parts (in  $\varepsilon$) of the momentum equation. Indeed, we get the following cascade of equations:

\noindent{\bf At order $\varepsilon^{-2},$} we get:
$$
\partial_{yy} u_0  = 0.
$$
The velocity field $u_0$ is therefore independent of the fast variable $y.$ 

\noindent {\bf At order $\varepsilon^{-1},$} we get then:
\begin{equation} \label{eq_nvonumero}
\rho^0 \left(  \partial_{\tau} u_0  + u_0 \partial_y  u_0 \right) = 2 \mu \partial_{xy} u_0  + \mu \partial_{yy} u_1  - \partial_y p^0.
\end{equation}
(we denote $\rho^0 = \theta_+ \rho_+^0 + \theta_-\rho_-^0$ and for the pressure: $p^0 = \theta_+ p_+^0 + \theta_-p_-^0$).
As $u_0$ does not depend on $y$ (and $\rho^0$ remains far from $0$), multiplying this equation by  $\partial_{\tau} u_0$ and integrating on a cell, we obtain $\partial_{\tau} u_0 = 0$. Therefore $u_0$ does not depend on both fast variables. In particular 
$$
\overline{u_0 \partial_x \theta_+}  = u_0 \partial_x \alpha_+\,, \quad \overline{\theta_+ u_0} = \alpha_+ u_0,
$$
and \eqref{eq_alpha2} rewrites:
\begin{equation} \label{eq_alpha22}
\partial_t (\alpha_+ \rho^0_+) + \partial_x (\alpha_+ \rho^0_+ u_0) = 0.
\end{equation}
Using  then that $u_0$ does not depend on the fast variables in \eqref{eq_nvonumero} gives, because $\mu$ is constant:
$$
\mu \partial_{yy} u_1 - \partial_y p^0 = 0, \quad \text{ and then }, \quad \mu \partial_y u_1 = p^0 - \overline{p^0}. 
$$
Multiplying this identity by $\theta_+$ and taking the average (we recall that
 $p_+^0 = p(\rho_+^0)$ and $p_-^0 = p(\rho_-^0)$ do not dependent of
 the fast variable), we get 
$$
-\overline{u_1 \partial_y \theta_+} =  \overline{\theta_+ \partial_y u_1} = \dfrac{1}{\mu} \overline{\theta_+ (p^0- \overline{p^0})} = \dfrac{\alpha_+ \alpha_-}{\mu} (p_+^0 - p^0_-).
$$
Finally, we obtain the expected equation for the volume fraction:
\begin{equation} \label{eq_alpha-12}
 \partial_{t} \alpha_+ + {u_0 \partial_x \theta_+} =   \dfrac{\alpha_+ \alpha_-}{\mu} (p_+^0 - p^0_-).  
\end{equation}

\noindent  {\bf At order $\varepsilon^{0}$,} in the momentum equation, we have now:
\begin{equation} \label{eq_momentum0}
\rho^0 \partial_t u_0 + \rho^0 u_0 \partial_x u_0 + \rho^0 \left( \partial_{\tau} u_1 + u_0 \partial_y u_1 \right)
= \partial_{x} \Sigma_0 + \partial_y \Sigma_1
\end{equation}
where 
$$
\Sigma_0 = \mu (\partial_x u_0 +  \partial_y u_1 ) - \sum_{i=+,-} \theta_i p(\rho^0_i)\,. 
$$
On the left-hand side, we recall that 
$$
\partial_y u_1 =  \dfrac{1}{\mu} \left( \sum_{i=+,-} \theta_i p^0_i - \overline{p}^0 \right) 
$$
and, in terms of the fast variables $(\tau,y)$, $\partial_y u_1$ is thus a linear function of $\theta_{\pm}$ only so that \eqref{eq_alpha-1} induces that 
$$
\partial_y (\partial_{\tau} u_1 + u_0 \partial_y u_1) = 0
$$ 
and consequently (because $\partial_{\tau} u_1 + u_0 \partial_y u_1$ has average $0$ on a cell):
$$
\partial_{\tau} u_1 + u_0 \partial_y u_1 = 0.
$$
Taking the average of \eqref{eq_momentum0} w.r.t fast variables, we obtain finally:
$$
\overline{\rho} \partial_t u_0 + \overline{\rho} u_0 \partial_x u_0 = \partial_{x} \overline{\Sigma}_0 
$$
with 
$$
\overline{\rho} = \alpha_+ \rho_+^0 + \alpha_- \rho_-^0\,, \quad \overline{\Sigma}_0 = \mu \partial_x u_0  - \sum_{i=+,-} \alpha_i p_i^0\,.
$$
Combining with \eqref{eq_alpha22}-\eqref{eq_alpha-12}, this completes the justification of the bifluid system in the constant-viscosity case.

\subsection{The density-dependent viscosity case} 
In this second case, we go back to the
relation 
\begin{equation} \label{eq_alpha01_mu}
 \partial_{t} {\theta}_+ + {u_0 \partial_x \theta_+} =  -{u_1 \partial_y \theta_+}. 
\end{equation}
that we want to average. We write again the different scales on the momentum equations.
We recall that we assume density-dependent viscosity $\mu = \mu(\rho)$. Therefore we can write 
$$
\mu = \theta_+ \mu_+^{\varepsilon} + \theta_- \mu_-^{\varepsilon}
$$
where we assume at first order that $\mu^{\varepsilon}_{\pm} \sim \mu^{0}_{\pm}$  which does not
depend on the fast variables.

\noindent {\bf Order $\varepsilon^{-2}$.} We get
$$
\partial_{y} [\mu \partial_{y} u_0 ]= 0
$$
This implies that 
$$
\mu \partial_y u_0 = K_0.
$$
To determine $K_0$ we use the equation  at order $\varepsilon^{-1}$ for $\theta_+$ (and $\theta_-$) \eqref{eq_alpha-1} that we multiply by  $\mu_+^{\varepsilon}$ (et $\mu_-^{\varepsilon}$ respectively). 
 After some combinations, we get:
$$
\partial_{\tau} \mu + u_0 \partial_y \mu  =0
$$
Averaging the equation on a cell, we get that 
$$
0 = \overline{u_0 \partial_y \mu } = - \overline{\mu \partial_y u_0} = K_0 
$$
Finally, we get that  $\partial_y u_0 = 0$ and therefore  $u_0$ does not depend of the space fast variables. 

\noindent {\bf Order $\varepsilon^{-1}$.}  We get with the same arguments as previously
$$ \rho^0 \partial_\tau u_0 = \partial_{y}[\mu^{0}(\partial_y u_1+\partial_x u_0)]  - \partial_y p^0$$
and therefore, because  $\partial_\tau u_0$ is constant, we obtain $\partial_{\tau} u_0=0$ after multiplication by $\partial_{\tau} u_0$
and integration on a cell. Hence, $u_0$ does not depend on both fast variables again and we obtain \eqref{eq_alpha22}. Also, the above equation then reduces to:
$$
0 = \partial_{y}[\mu^{0}(\partial_y u_1+\partial_x u_0)]  - \partial_y p^0.
$$
This gives 
\begin{equation} \label{eq_u12}
\mu^{0} \partial_y u_1 + (\mu^{0} - \overline{\mu^0})\partial_x u_0 - (p^0 - \overline{p}^0) = K_1
\end{equation}
Let us note that at first order $K_1 = \overline{\mu^0 \partial_y u_1}.$ 
We want to calculate this quantity. To calculate  $K_1,$ we proceed as previously, we multiply  \eqref{eq_alpha01_mu} by $\mu_+^{0}$ (and its equivalent for  $\theta_-$ by $\mu_{-}^{0}$). 
After some combinations, this gives:
$$
 \partial_{t} \mu^{0}  + {u_0 \partial_x \mu^{0}} + {u_1 \partial_y \mu^{0}}  = \sum_{i=\pm} \theta_i \left(\partial_t + u_0 \partial_x \right) \mu^0_i .
$$
Averaging with respect to the fast variable, we get
\begin{eqnarray*}
\overline{u_1 \partial_y \mu^{0}} &=&  \sum_{i=\pm} \alpha_i \left(\partial_t + u_0 \partial_x \right) \mu^0_i -   \left(\partial_t + u_0 \partial_x \right) \overline{\mu^0} \\
&=& - \sum_{i=\pm} \mu_i^0 \left( \partial_t + u_0 \partial_x \right) \alpha_i \\
&=& (\mu^0_- - \mu^0_+) \left( \partial_t + u_0 \partial_x \right) \alpha_+ 
\end{eqnarray*}
and finally
\begin{equation} \label{eq_dyu1average}
K_1 = \overline{\mu^{0} \partial_y u_1 } = (\mu^0_+ - \mu^0_-) \left( \partial_t + u_0 \partial_x \right) \alpha_+ 
\end{equation}

Thus we can calculate at first order:
\begin{eqnarray*}
- \overline{u_1 \partial_y \theta_+} &= & \overline{\theta_+ \partial_y u_1} \\
		&=&   \overline{      \dfrac{\theta_+ }{\mu^0} (p^0 - \overline{p^0}) - \dfrac{\theta_+ }{\mu^0}(\mu^0 - \overline{\mu^0}) \partial_x u_0 +  \dfrac{\theta_+}{\mu^0} K_1 } .
\end{eqnarray*}
Then we have 
\begin{eqnarray*}
\overline{ \dfrac{\theta_+ }{\mu^0} (p^0 - \bar{p}^0) } &=& \dfrac{\alpha_+}{{\mu}^0_+} p^0_+ - \dfrac{\alpha_+}{\mu^0_+} \left(\alpha_+ p_+^0  + \alpha_- p_-^0 \right) \\  
&=& \dfrac{\alpha_+ \alpha_-}{\mu^0_+} (p^0_+ - p^0_-),
\end{eqnarray*}
and on the other part
$$
\overline{ \dfrac{\theta_+}{\mu^0} K_1 } = \dfrac{\alpha_+}{\mu^0_+} (\mu^0_+ - \mu^0_-) \left( \partial_t + u_0 \partial_x \right) \alpha_+ ,
$$
and finally : 
$$
 \overline{      \dfrac{\theta_+ }{\mu^{0}} (\mu^0 - \bar{\mu}^0) \partial_x u_0 } = \partial_x u_0 \dfrac{\alpha_+ \alpha_-}{\mu_+^0}   (\mu^0_+ - \mu^0_-).
$$
Thus
$$
- \overline{u_1 \partial_y \theta_+}  = \dfrac{\alpha_+ \alpha_-}{\mu^0_+} \left( (p^0_+ - p^0_-) - \partial_x u_0 (\mu^0_+ - \mu^0_-) \right) +  \alpha_+ \left(  1- \dfrac{\mu^0_-}{\mu_+^0} \right) \left( \partial_t + u_0 \partial_x \right) \alpha_+ 
$$
Therefore we get finally the following equation on  $\alpha_+:$
\begin{equation}
\left(1 + {\alpha}_+ \left(\dfrac{\mu_-^0}{\mu_+^0} -1 \right) \right) \left( \partial_t {\alpha}_+ + u_0 \partial_x {\alpha}_+ \right) = \dfrac{{\alpha}_+ {\alpha}_-}{\mu_+^0} \left(  (p_+^0 - p_-^0) - \partial_x u_0 (\mu_+^0 - \mu_-^0)\right)
\end{equation}
which may be rewritten as:
\begin{equation} \label{eq_alpha23}
\partial_t \alpha_+ + u_0 \partial_x \alpha_+ = \dfrac{\alpha_+ \alpha_-}{\alpha_+ \mu_-^0  + \alpha_- \mu_+^0} \left(  (p_+^0 - p_-^0 )- \partial_x u_0 (\mu_+^0 - \mu_-^0)\right).
\end{equation}
   As for the momentum equation, we write the $\varepsilon^{0}$ order of the momentum equation as in the previous case. We remark again that, thanks to \eqref{eq_u12}, the quantity $\partial_{y}u_1$ depends on the fast variable only through $\theta_{\pm}$ so that after averaging, we obtain:
$$
\overline{\rho} \partial_t u_0 + \overline{\rho} u_0 \partial_x u_0 = \partial_{x} \overline{\Sigma}_0 
$$
with 
$$
\overline{\rho} = \alpha_+ \rho_+^0 + \alpha_- \rho_-^0\,, \quad \overline{\Sigma}_0 = \mu^0 \partial_x u_0  - \sum_{i=+,-} \alpha_i p_i^0\, + \overline{\mu^0 \partial_y u_1}.
$$
Combining \eqref{eq_dyu1average} and \eqref{eq_alpha23}, we have:
$$
\overline{\mu^0 \partial_y u_1} = \dfrac{\alpha_+ \alpha_- (\mu_+^0 - \mu_-^0)}{\alpha_+ \mu_-^0  + \alpha_- \mu_+^0} \left(  (p_+^0 - p_-^0 )- \partial_x u_0 (\mu_+^0 - \mu_-^0)\right),
$$
so that, after tedious but straightfoward algebraic combination (using many times that $\alpha_++ \alpha_- = 1$), we get:
$$
\overline{\Sigma}_0 = \dfrac{\mu^0_+ \mu^0_-}{\alpha_+ \mu_-^0 + \alpha_- \mu_+^0 } \partial_x u_0 - \dfrac{\alpha_+ p_+^0 \mu_-^0 + \alpha_- p_-^0 \mu_+^0}{\alpha_+ \mu_-^0 + \alpha_- \mu_+^0 }.
$$
This completes the justification of the bifluid system mentioned in the introduction.


\section{Mathematical proofs.}
   In this section, we mathematically justify the derivation of a multifluid system with variable viscosities from the compressible Navier--Stokes system with a  density-dependent viscosity.
   We consider the one-dimensional in space case to be able to construct global strong solutions
 far from vacuum in the classical setting. We therefore recall and make precise the result of existence and method of proof
 coming from \cite{MeVa} and recently \cite{Ha}. Then,  we perform the compactness  result and derive the multifluid system.
  
\subsection{Strong solution theory}
By adapting the arguments of \cite{MeVa} to our periodic framework,  we have the following existence theorem
\begin{theorem} \label{thm_strsol}
Given $\rho^0 \in H^1_{\sharp}$ and $u^0 \in H^1_{\sharp}$ satisfying
$$
\underline{\rho^0} := \inf \rho^0(x) >0,
$$
there exists a unique pair $(\rho,u)$ such that:
\begin{itemize}
\item[$\mathbf{(CS)}_a$] we have the regularity statement 
\begin{align}
 & \rho \in C([0,\infty) ; H^1_{\sharp}) \text{ with } \rho > 0\,,  \\
 & u \in C([0,\infty);H^{1}_{\sharp})  \cap L^2_{loc} ((0,\infty) ; H^2_{\sharp}) \,;     
 \end{align}
\item[$\mathbf{(CS)}_b$] $(\rho,u)$ satisfies \eqref{eq_NSC} a.e. in $(0,\infty) \times \mathbb R$ with $p,\mu$ given 
by \eqref{eq_ee} ;
\item[$\mathbf{(CS)}_c$] $(\rho,u)$ matches initial conditions \eqref{eq_id} a.e..
\end{itemize} 
\end{theorem}
We sketch the proof of this theorem for completeness.

Local existence of solutions is obtained by a classical fixed-point argument so that the only difficulty lies in proving these solutions are global. As the local-in-time theory
yields a time of existence depending only on 
$$
\mathcal E(0) := \underline{\rho^0} + \|\rho^0 \|_{H^1_{\sharp}} + \|u^0 \|_{H^1_{\sharp}},
$$
we aim to obtain a local-in-time uniform bound on $\mathcal E(t)$ for the associated solution $(\rho,u)$.
This solution is defined {\em a priori} on a non-extendable time interval $[0,T_*).$

{\bf Step 1. Dissipation of energy.} First, with classical arguments, we obtain 
\begin{multline} \label{eq_est0}
\int_{0}^L  \left[ \dfrac{\rho(t,x) |u(t,x)|^2   }{2}  {\rm d}x+  q(\rho(t,x)) \right] + \int_0^t\int_{0}^L \mu(t,x) |\partial_x u(s,x)|^2  {\rm d}x{\rm d}t  \\
= \int_{0}^L  \left[ \dfrac{\rho^0(x) |u^0(x)|^2   }{2}  {\rm d}x+  q(\rho^0(x)) \right]{\rm d}x
\end{multline}
for all $t \in [0,T_*)$ where $q$ is defined by:
\begin{equation} \label{eq_defq}
q(z) = z \partial_{s}^{-1}\left\{\dfrac{p(s)}{s^2} \right\}.  
\end{equation}

\bigskip

{\bf Step 2. BD entropy.} We control now the growth of the $H^1$-norm of $\rho.$ 
 Namely, we adapt to our periodic case the BD-entropy method which may be found
in its simplest form in \cite{BrDeEvian}  for instance. 
    So, we introduce  $\varphi \in C^{1}((0,\infty))$ defined by 
$$
\varphi(z) =  \int_{1}^{z} \dfrac{\mu(s)}{s^2} {\rm d}s \,, \quad \forall \, z \in (0,\infty)\,.
$$
Note that for a  nonlinear function $\varphi_1$ of the density, we have 
$$\partial_t  \varphi_1(\rho) +  u\partial_x \varphi_1(\rho)   + \varphi_1'(\rho)\rho  \partial_x u = 0.$$
Thus differentiating with respect to space    
$$\partial_t  \partial_x(\varphi_1(\rho)) +  
\partial_x(u\partial_x \varphi_1(\rho) )  +  
\partial_x(\varphi_1'(\rho)\rho  \partial_x u )= 0.$$                                
Let us now choose 
$\varphi_1(\rho) = \int_1^\rho  \mu(\rho)/\rho $, then we get
from the definition of $\varphi$
$$\partial_t  (\rho \partial_x\varphi(\rho)) +  
\partial_x(\rho u\partial_x \varphi(\rho) )  +  
\partial_x(\mu(\rho)  \partial_x u )= 0.$$  
Adding the relation to the momentum equation gives
\begin{equation}\label{eqq}
\partial_t  (\rho (u+ \partial_x\varphi(\rho) ))+  
\partial_x(\rho u (u+ \partial_x \varphi(\rho)) )  +  
\partial_x p(\rho) = 0.
\end{equation} 
   In what follows, we denote $\varphi_x := \partial_x \varphi(\rho(x))$ to be distinguished with $z \to \varphi'(z)$ the simple
derivative of the above defined function $\varphi.$ We keep subscript $x$ to denote partial derivatives w.r.t. space variable (we have thus $\partial_x u = u_x$).  
Testing the equation \eqref{eqq} with $u+\partial_x \varphi$ yields finally:
\begin{equation} \label{eq_borneBD}
\dfrac 12 \dfrac{\textrm d}{\textrm{d}t} \left[ \int_{0}^L \left\{ \rho \left| u + \varphi_x \right|^2 +  q(\rho) \right\} \right] + \int_{0}^L p'\varphi' |\rho_x|^2  
=0\,.
\end{equation}
As $p'\varphi' \geq 0$ we conclude that  
$$
\int_{0}^L \left\{ \rho \left| u + \varphi_x \right|^2 +  q(\rho) \right\} \leq C_0\,, \quad \forall \, t \geq 0\,.
$$
Hence:
\begin{equation} \label{eq_BD2}
\int_0^L |\sqrt{\rho} \varphi'(\rho) \rho_x|^2 \leq C_0\,, \quad \forall \, t \geq 0. 
\end{equation}
As the continuity equation implies the conservation of the mean of $\rho$ on $(0,L)$ we derive that, setting
$f \in C^1((0,\infty))$ any primitive of $z \mapsto \mu(z)/z^{3/2}$, there holds:
$$
\|f(\rho(t,\cdot))\|_{ L^{\infty}_{\sharp}} \leq C_0\,, \quad \forall \, t \geq 0\,.
$$ 
In particular, our assumption \eqref{eq_assummu} on $\mu$  enforces that $f(z)$ diverges when $z \to 0$ or $z \to \infty.$
Hence, we obtain from the control above that
\begin{equation} \label{eq_BD3}
\|\rho(t,\cdot) \|_{L^{\infty}_{\sharp}} + \|\rho^{-1}(t,\cdot) \|_{L^{\infty}_{\sharp}}\leq C_0\,, \quad \forall \, t \geq 0\,,
\end{equation}
and, plugging this inequality into \eqref{eq_BD2} (and applying again that the mean of $\rho$ is constant with time so that
the $\|\partial_x \rho\|_{L^2_{\sharp}}$ controls the $H^1$-norm of $\rho$), we get:
\begin{equation} \label{eq_BD4}
\|\rho(t,\cdot) \|_{H^1_{\sharp}} \leq C_0\,, \quad \forall \, t \geq 0 \,.
\end{equation}
From the BD-entropy argument we developed up to now, we obtain global-in-time control on the $\rho$
in the $H^1$-norm and in the $L^{\infty}$-norm from above and from below.

\noindent{\bf Remarks.}

\begin{enumerate}
\item[{\em 1.}] In case $p$ merely satisfies \eqref{eq_p2}, equation \eqref{eq_borneBD} induces that for a constant $C_{p\mu} >0$ there holds:
\begin{equation} 
\dfrac 12 \dfrac{\textrm d}{\textrm{d}t} \left[ \int_{0}^L \left\{ \rho \left| u + \varphi_x \right|^2 +  q(\rho) \right\} \right]  \leq  C_{p\mu} \int_{0}^L \rho|\varphi_x|^2  \,.
\end{equation}
Hence, recalling that the total energy of the solution remains uniformy bounded with time, we obtain, by applying a Gronwall lemma, that there exists
a positive constant $C_0$ depending only on initial data, for which:
$$
\int_{0}^L \left\{ \rho \left| u + \varphi_x \right|^2 +  q(\rho) \right\} \leq C_0C_{p\mu}(1+t)\exp(2C_{p\mu}t) \quad \forall \, t \geq 0\,.
$$

\item[{\em 2.}]  Recently, B. {\sc Haspot} has extended the range of viscosity
that provides global existence of strong solution for the compressible Navier-Stokes equation with
density-dependent viscosity if initially the density is far from vacuum.
 His nice idea is to remark that the equation on $v= u+\partial_x \varphi(\rho)$ contains a damping term if we replace  the pressure term in terms of the $v$ and $u$.
More precisely, we get the equation
\begin{equation}\label{ree}
\partial_t  (\rho v)+  
\partial_x(\rho u v  )  +  \frac{p'(\rho)\rho^2}{\mu(\rho)} v = \frac{p'(\rho)\rho^2}{\mu(\rho)} u.
\end{equation} 
Thus for $p(s)=a s^\gamma$ ($\gamma>1$) if we assume $\mu(s) \le C + C p(s)$ for all $s\ge 0$ then he first proves that $v$ is $L^\infty (0,T;L^\infty)$ and then coming back to the
mass equation that $1/\rho$ belongs to $L^\infty$. This allows him to extend a local in time
result to a global one. In conclusion, our homogenized result may for instance be extended to the shallow-water  system where $\mu(\rho)= \rho$ and $p(\rho) = a\rho^2$.
\end{enumerate}

\bigskip

{\bf Step 3. Regularity.} Finally, we obtain propagation of the $H^1$ regularity for $u.$ Namely, we differentiate
once the momentum equation, yielding:
$$
\rho (\partial_t u_x + u \partial_x u_x ) = \partial_{x} \left( \mu \partial_x u_x  \right) - p_{xx} - \rho_x \left( \partial_t u + u \partial_x u\right) - \rho |u_{x}|^2 + \partial_x(\mu_x u_x).
$$
Multiplying this equality with $u_x$  yields:
$$
\dfrac 12 \dfrac{\textrm{d}}{\textrm{d}t} \left[\int_{0}^L \rho |u_x|^2 \right] + \int_0^{L} \mu |\partial_x u_{x}|^2
=  - \int_{0}^L  \left(p_{xx} + \rho_x \left( \partial_t u + u \partial_x u\right) + \rho |u_{x}|^2 \right) u_x -  \int_{0}^L \mu_x \partial_x u_x u_x .
$$
On the right-hand side, we have after integration by parts:
$$
\left| \int_0^L  p_{xx} u_x \right| \leq \dfrac{C}{\mu^0 \varepsilon} \|p \|_{H^1(0,L)}^2 + \varepsilon \int_0^L \mu |\partial_x u_x|^2. 
$$
Then, we replace 
$$
\rho_x \left( \partial_t u + u \partial_x u\right)  = \dfrac{\rho_x}{\rho} \left[ \mu u_{xx} + \mu_x u_x  - p_x \right]
$$
so that:
\begin{eqnarray*}
\left|\int_{0}^L   \rho_x \left( \partial_t u + u \partial_x u\right)  u_x \right|  & \leq  & 
   \|\mu\|_{ L^{\infty}_{\sharp}}^{\frac 12} \|\rho^{-1} \|_{L^{\infty}_{\sharp}} \|\rho\|_{ H^1_{\sharp}}
    \|u_x \|_{ L^{\infty}_{\sharp}} \left( \int_0^L \mu |u_{xx}|^2 \right)^{\frac 12} \\
&& + \|\rho^{-1} \|_{L^{\infty}_{\sharp}}
\|\rho\|_{H^1_{\sharp}} \|\mu_x \|_{L^2_{\sharp}} 
 \|u_x \|^2_{L^{\infty}_{\sharp}} \\
&&+  \|\rho^{-1} \|_{L^{\infty}_{\sharp}}
\|\rho \|_{H^1_{\sharp}}  
\|p_x\|_{L^{2}_{\sharp}} \|u_x\|_{L^{\infty}_{\sharp}}\\
& \leq  &  C_0 \left( 1 + \|u_x \|_{L^2_{\sharp}}^2 \right) + \varepsilon \int_0^L \mu |u_{xx}|^2 
\end{eqnarray*}
where we applied the previous controls on $\rho$ in the $H^1$ and $L^{\infty}$ norms, and that, for an absolute constant $C$, there holds:
$$
\|u_x \|_{L^{\infty}_{\sharp}} \leq C \|u_x \|_{L^{2}_{\sharp}}^{\frac 12}\|u_{xx}\|_{L^{2}_{\sharp}}^{\frac 12}\,.
$$
We have similarly:
$$
\left| \int_0^L  \rho |u_x|^2 u_x \right| \leq \dfrac{C_0}{2\varepsilon}(1+ \|u_x \|_{L^2_{\sharp}}^4) + \varepsilon   \int_0^L \mu |u_{xx}|^2.
$$
and 
$$
\left| \int_0^L  \mu_x u_x \partial_x u_x \right| \leq \dfrac{C_0}{\varepsilon} \|u_x \|_{L^{2}_{\sharp}}^{2} + \varepsilon   \int_0^L \mu |u_{xx}|^2
$$
Combining all these computations in our first identity, and choosing $\varepsilon$ sufficiently small, yields:
$$
\dfrac 12 \dfrac{\textrm{d}}{\textrm{d}t} \left[\int_{0}^L \rho |u_x|^2 \right] + \frac 12 \int_0^{L} \mu |\partial_x u_{x}|^2
=  C_0 \left( 1 + \|u_x \|^4_{L^2_{\sharp}} \right)
$$
Applying a standard Gronwall inequality and recalling the dissipation of energy estimate, we obtain then that:
$$
\int_{0}^L |u_x(t,x)|^2{\rm d}x \leq C_0 (1+t) \exp(C_0)\,, \quad \forall \, t \geq 0
$$
This ends the proof.

\subsection{Uniform estimates}
Prior to establishing {\bf Theorem \ref{thm_existence1}}, we show in this section that the global strong solutions of the previous section, that we construct  for initial data $(\rho^0,u^0) \in H^1_{\sharp} \times H^1_{\sharp},$ 
do satisfy the requirements ${\textbf{(HDS)}}_a,$  ${\textbf{(HDS)}}_b$ and ${\textbf{(HDS)}}_c$ of {\bf Theorem \ref{thm_existence1}} on some time interval $(0,T_0)$ with $T_0$ depending only on $\underline{\rho^0},\overline{\rho^0},$
$\|u^0\|_{H^1_{\sharp}}.$ This completes the proof of {\bf Theorem \ref{thm_existence1}} in the case $\rho^0$ satisfies
the further property $\rho^0 \in H^1_{\sharp}.$

So, let $(\rho^0,u^0) \in H^1_{\sharp} \times H^1_{\sharp}$ and $(\rho,u)$ the associated global strong solution 
given by {\bf Theorem \ref{thm_strsol}}. Clearly, ${\textbf{(CS)}}_a$ (resp. ${\textbf{(CS)}}_b$) induces that ${\textbf{(HDS)}}_a$  (resp. ${\textbf{(HDS)}}_b$) holds on arbitrary time-interval $(0,T_0)$. 
We remind also that this solution satisfies the dissipation energy estimate
\eqref{eq_est0}. Hence, denoting by
$$
\mathcal E^c_0 := \int_{0}^L  \left[ \dfrac{\overline{\rho^0} |u^0(x)|^2   }{2}  {\rm d}x+  \max_{[\underline{\rho^0},\overline{\rho^0}]} q(z) \right]{\rm d}x
$$
we have that, for arbitrary $T_0>0$ :
\begin{equation} \label{eq_ec0}
\sup_{t \in (0,T_0)} \left[  \dfrac{1}{2}\int_0^L \rho(t,x) |u(t,x)|^2 {\rm d}x + \int_0^{t} \int_0^L \mu(s,x)|\partial_x u(s,x)|{\rm d}s {\rm d}x \right] 
\leq  \mathcal E_c^0. 
\end{equation}
The only point is thus to obtain the bounds  \eqref{eq_regathm} and \eqref{eq_regcthm}. 
Note also that thanks to the regularity ${\textbf{(CS)}}_a$, these conditions are indeed satisfied but for a sufficiently small $\tilde{T}_0$ only.  The actual difficulty is thus to prove that we may choose $\tilde{T}_0 = T_0$ depending only on $\underline{\rho^0},\overline{\rho^0},\|u^0\|_{H^1_{\sharp}}$.   For this purpose, in what follows, we pick a positive time $\tilde{T}_0$ for which \eqref{eq_regathm}
and \eqref{eq_regcthm} are satisfied by $(\rho,u)$ on  $[0,\tilde{T}_0]$ for a well chosen $K^0.$ We show then, that, if we assume $\tilde{T}_0 < T_0,$ for some $T_0$ to be constructed with the expected dependencies, we obtain a better bound for $(\rho,u).$ By a standard connectedness argument\footnote{
Given the regularity statements $\textbf{(CS)}_{a}$  the following quantities are continuous functions of time-variable $t \in [0,\infty)$:
$$
\min_{[0,L]} \rho(t,x)\,, \quad \max_{[0,L]} \rho(t,x)\,,  \quad \int_0^t \|\partial_x z \|_{L^2_{\sharp}}^{2} \,, \quad \sup_{(0,t)} \|u\|_{H_{\sharp}^1}
$$    }, we obtain then that  we may choose $\tilde{T}_0 = T_0.$  

For the computations below, we introduce the following notations:
\begin{itemize}
\item we introduce the function $\kappa =p/\mu$\\
\item  given $\beta \in C([0,\infty)),$ we denote
$$
K^0_{\beta} = \max \{ \beta(z), z \in [\underline{\rho^0}/2,2\overline{\rho^0}] \}
$$
 \item $K^0_u =  36 \left( \dfrac{1}{\mu^0} +  \overline{\rho^0} \right)   \left[ \|\sqrt{\mu(\rho^0)} \partial_x u _0- \kappa(\rho^0) \|_{L^2_{\sharp}}^2 + 1 + L|K^0_{\kappa}|^2\right].$
\end{itemize}
We remark that $K^0_u$ do depend only on 
$\overline{\rho^0},\underline{\rho^0},\|u^0\|_{H^1_{\sharp}}.$ It will play the role of $K^0$ in our proof.

According to the method of proof we described above, we assume from now on that $\tilde{T}_0>0$ is chosen and fixed such that
we have the a priori bounds: 
\begin{equation} \label{eq_Rega}
\dfrac{\underline{\rho^0}}{2} \leq \rho(t,x) \leq  2\overline{\rho^0} \qquad \text{ on $(0,\tilde{T}_0) \times \mathbb R$}
\end{equation}
\begin{equation}\label{eq_Regb}
 \sup_{(0,\tilde{T}_0)} \|u\|_{H^1_{\sharp}}^2  + \int_0^{\tilde{T}_0} \| \partial_x z\|_{L^2_{\sharp} }^2 {\rm d}s \leq K^0_u.
\end{equation}

We state first the following  lemma:
\begin{lemma} \label{lem_divu}
Let denote: 
$$
K^0_{d} =  \dfrac 1{\mu^0}\left( \sqrt{L} + \dfrac{1}{\sqrt{L}} \right)  \left(2  K_{\mu}^0 \mathcal E_{0}^c + 2 L |K^0_p|^2 + K_u^0 \right)^{\frac 12} + \dfrac{K_{p}^0}{\mu^0},
$$
(see \eqref{eq_assummu} for the definition of $\mu^0$). Then, $K^0_d$ depends only on $\underline{\rho^0},\overline{\rho^0},\|u^0\|_{H^1_{\sharp}}$ and, if $\tilde{T}_0 < 1,$ there holds
\begin{equation} \label{eq_Regd}
\int_0^{\tilde{T}_0}\|\partial_x u \|_{L^{\infty}_{\sharp}} \leq |{\tilde{T}_0}|^{\frac 12} K^0_d.
\end{equation}
\end{lemma}
\begin{proof}
We recall first the classical embedding 
$H^1_{\sharp} \subset L^{\infty}_{\sharp}$ with the embedding inequality:
$$
\|v \|_{L^{\infty}_{\sharp}} \leq \left( \sqrt{L} + \dfrac{1}{\sqrt{L}} \right) \|v\|_{H^1_{\sharp}},
$$
Let now $T \leq \tilde{T}_0.$ Due to \eqref{eq_Regb}, we have
$$ 
\int_0^T \int_0^L |\partial_x z|^2 \leq K^0_u.
$$
Then, by construction, there holds:
$$
|z|^2 \leq 2 \left(  |\mu|^2 |\partial_x u|^2 + |p|^2\right)
$$
Consequently, recalling \eqref{eq_ec0}, we obain:
\begin{eqnarray*}
\int_0^T \int_0^L |z|^2 &\leq& 2 \int_0^T \int_0^L  |\mu|^2 |\partial_x u|^2 + 2 \int_0^T \int_0^L |p|^2 \\
							&\leq & 2 K_{\mu}^0  \int_0^T \int_0^L  \mu |\partial_x u|^2 + 2 TL |K^0_p|^2 \\
							& \leq & 2  K_{\mu}^0 \mathcal E_{0}^c + 2 TL |K^0_p|^2.
\end{eqnarray*}
Finally, we have:
$$
\int_0^T \|z \|_{ L^{\infty}_{\sharp}}^2 \leq \left( \sqrt{L} + \dfrac{1}{\sqrt{L}} \right)^2  \left(2  K_{\mu}^0 \mathcal E_{0}^c + 2 TL |K^0_p|^2 + K^0_u \right).
$$
and thus
$$
\int_0^T \|z \|_{L^{\infty}_{\sharp}} \leq \sqrt{T}\left( \sqrt{L} + \dfrac{1}{\sqrt{L}} \right)  \left(2  K_{\mu}^0 \mathcal E_{0}^c + 2 TL |K^0_p|^2 + K^0_u \right)^{\frac 12}.
$$
Then, we remark that 
$$
\partial_{x} u = \dfrac{z+ p }{\mu},  \quad \text{ so that (with the bound \eqref{eq_assummu}),}  \quad |\partial_x  u| \leq \dfrac{1}{\mu^0} |z| + \dfrac{p}{\mu^0}
$$
and :
$$
\int_0^T \|\partial_x u\|_{L^{\infty}_{\sharp}}  \leq \dfrac{\sqrt{T}}{\mu^0}\left( \sqrt{L} + \dfrac{1}{\sqrt{L}} \right)  \left(2  K_{\mu}^0 \mathcal E_{0}^c + 2 TL |K^0_p|^2 + K^0_u \right)^{\frac 12} + \dfrac{T}{\mu^0} K_{p}^0.
$$
Hence, under the further restriction $T < 1,$ we obtain:
$$
\int_0^T \|\partial_x u \|_{ L^{\infty}_{\sharp}}  \leq \dfrac{\sqrt{T}}{\mu^0} \left[ \left( \sqrt{L} + \dfrac{1}{\sqrt{L}} \right) \left(2  K_{\mu}^0 \mathcal E_{0}^c + 2 L |K^0_p|^2 + K^0_u \right)^{\frac 12} +  K_{p}^0\right]
$$
which yields the expected result setting $T=\tilde{T}_0.$
\end{proof}

We now consider the continuity equation and derive bounds for $\rho$:
\begin{lemma} \label{lem_rho}
There exists $T^{\rho}_0$ depending only on $\underline{\rho^0},\overline{\rho^0},\|u^0\|_{H^1_{\sharp}}$ for which, if we assume that  $\tilde{T}_0 < T^{\rho}_0$ then, there holds:
$$
\dfrac{2}{3} \underline{\rho^0} < \rho(t,x) < \dfrac{3}{2} \overline{\rho^0} \quad \forall \, (t,x) \in (0,\tilde{T}_0) \times \mathbb R.
$$ 
\end{lemma}
\begin{proof}
By standard arguments, we have that, for arbitrary $p \in [1,\infty[ \cup ]-\infty,-1[$ there holds:
$$
\dfrac{1}{p}\dfrac{\textrm{d}}{{\textrm d}t} \left[ \int_{0}^L |\rho|^p \right]  + \dfrac{1}{p} \int_0^L u \partial_x |\rho|^p = - \int_0^L |\rho|^p \partial_x u
$$
so that:
$$
\dfrac{\textrm{d}}{{\textrm d}t} \left[ \int_{0}^L |\rho|^p \right]  \leq   \left| p-1  \right|\int_0^L |\rho|^p |\partial_x u|
\leq 2|p| \|\partial_x u \|_{L^{\infty}_{\sharp}}  \int_0^L |\rho|^p.
$$
Consequently, there holds:
$$
\left( \int_0^L |\rho|^p \right)^{\frac 1p}\leq  \left( \int_0^L |\rho^0|^p \right)^{\frac 1p} \exp\left( 2 \int_0^T \|\partial_x u \| _{L^{\infty}_{\sharp}}\right).
$$
In the limit $|p| \to \infty$ we thus have
\begin{align*}
&\|\rho(t,\cdot)\|_{L^{\infty}_{\sharp}} \leq \overline{\rho^0} \exp\left(2 \int_0^T \|\partial_x u\|_{L^{\infty}_{\sharp}}\right)\,, \\
&\||\rho(t,\cdot)|^{-1} \|_{L^{\infty}_{\sharp}} \leq \dfrac{1}{\underline{\rho^0}}\exp\left(2 \int_0^T
 \|\partial_x u \|_{L^{\infty}_{\sharp}}\right).
\end{align*}
Setting 
$$
T^{\rho}_0 :=  \min\left(\frac 12 , \left| \dfrac{1}{2K^0_d} \ln\left(\frac 32\right) \right|^2\right)
$$
(which has then the expected dependencies, see the definition of $K^0_d$), 
and assuming $\tilde{T}_0 < T^{\rho}_0 < 1,$ we apply Lemma \ref{lem_divu} on $(0,\tilde{T}_0)$ and obtain:
$$
\underline{\rho^0} \exp\left(-2|\tilde{T}_0|^{\frac 12} K_d^0 \right) \leq \rho(t,x) \leq  \overline{\rho^0}  \exp\left(2 |\tilde{T}_0|^{\frac 12} K_d^0 \right)\,,
$$
where
$$
\exp(2 |\tilde{T}_0|^{\frac 12}K_d^0) \leq  \exp(2 \sqrt{T^{\rho}_0} K_d) \leq \dfrac{3}{2}\,.
$$
\end{proof}

We conclude with deriving estimates for $u$ and $z:$
\begin{lemma} \label{lem_uniformbound}
There exists $T^{u}_0>0$ depending only on $\underline{\rho^0},\overline{\rho^0},\|u^0\|_{H^1_{\sharp}}$ for which, if we assume that  $\tilde{T}_0 < T^{u}_0$, there holds:
\begin{multline*}
\sup_{t \in [0,\tilde{T}_0]} \| \partial_x u \|^2_{L^2_{\sharp}}  + \int_0^{\tilde{T}_0} |\partial_x z |^2{\rm d}s   \\
\leq  \left( \dfrac{8}{\mu^0} + 36 \overline{\rho^0} \right)   \left[ \|\sqrt{\mu(\rho^0)} \partial_x u _0- \kappa(\rho^0) \|_{L^2_{\sharp}}^2 + 1 + L|K^0_{\kappa}|^2\right].
\end{multline*}
\end{lemma}
\begin{proof}
The proof of this result is based on the use of a suitable multiplier for the momentum equation:
$$
\rho (\partial_t u + u \partial_x u ) = \partial_x \left[ \mu \partial_x u - p \right]  \quad \text{a.e. on $(0,\infty) \times (0,L)$} 
$$  
Precisely,  we introduce the following conventions: 
\begin{enumerate}[$\bullet$]
\item the operator $\mathbb E$ corresponds to the mean of an $L$-periodic $L^1$-function;
\item the operator $\partial_{x}^{-1}$ corresponds to the periodic mean-free primitive
of an $L$-periodic function of mean $0.$ It maps $H^{m}_{\sharp}$ into $H^{m+1}_{\sharp}$ for arbitrary
$m \in \mathbb N$ and admits a straightforward density extension for $m\in \mathbb Z \setminus \mathbb N$ (when $m\in \mathbb Z\setminus \mathbb N$, $H^m_{\sharp}$ stands for the dual of the subspace of $H^{|m|}_{\sharp}$ containing all functions with mean zero);
\item throughout the proof, $C_0$ is a constant which depends only on $\underline{\rho^0},\overline{\rho^0}$ and $\|u^0\|_{H^1_{\sharp}}$. It may vary between lines.
\end{enumerate}
Then, we let $T \in (0,\tilde{T}_0)$ and we set:
$$
v =  \partial_t u - \partial_x^{-1} \left[ \partial_t \kappa  -  \mathbb E[\partial_t \kappa]  \right] \quad \text{ on $(0,T).$}
$$
We recall that thanks to the continuity equation, there holds, for arbitrary $\beta \in C^{1}([0,\infty))$
\begin{equation}
\partial_t \beta(\rho)  + \partial_x (\beta(\rho) u) + (\beta'(\rho) \rho - \beta(\rho))\partial_x u = 0. \label{eq_young}
\end{equation}
Hence we have that  $\kappa,\mu,p$ belong to the space $H^1(0,T;L^2_{\sharp}) \cap C([0,T] ; C_{\sharp}) \cap L^2(0,T;H^1_{\sharp}).$  Consequently, $v \in L^2((0,T) ; L^2_{\sharp})$ and we have then:
\begin{equation} \label{eq_vf1}
\int_0^T \int_0^L \rho (\partial_t u + u \partial_x u ) v = \int_0^T \int_0^L \partial_x \left[ \mu \partial_x u - p \right] v
\end{equation}
On the right-hand side, we note that we can approximate $u$ by projecting on Fourier series with a finite number of terms. 
This yields a sequence $u^N$ converging to $u$ in  $H^1(0,T;L^2_{\sharp}) \cap C([0,T] ; H^1_{\sharp}) \cap L^2(0,T ; H^2_{\sharp}).$ Furthermore, the extension of $\partial_{x}^{-1}$ to negative sobolev spaces yields that:
$$
v  = \partial_x^{-1} \left[ \partial_t \partial_x u -  (  \partial_t \kappa  -  \mathbb E[\partial_t \kappa]  ) \right]\,.
$$
Hence  the following formal integration by parts that are valid for $L$-periodic trigonometric polynoms (with $v^N = \partial_t u^N - \partial_x^{-1} \left[ \partial_t \kappa  -  \mathbb E[\partial_t \kappa]  \right] $):
\begin{eqnarray*}
  \int_{0}^L   \partial_x \left[ \mu \partial_x u^N - p  \right]  v^N &=& - \int_{0}^L    \left[ \mu \partial_x u^N - p  \right]  \partial_x v^N \\
		&=& - \int_{0}^L \mu \left[ \partial_x u^N - \kappa \right] \partial_t \left[ \partial_x u^N - \kappa \right] - \int_0^L\left[ \mu \partial_x u^N - p  \right]  \mathbb E[\partial_t \kappa]  \\
		&=& - \dfrac{1}{2} \dfrac{\textrm{d}}{\textrm{d}t}\int_0^L \mu \left| \partial_x u^N - \kappa \right|^2 +\dfrac{1}{2} \int_0^L \partial_t \mu \left| \partial_x u^N - \kappa \right|^2 \\
		&& \quad -\mathbb E[{\mu \partial_x u^N - p}] \int_0^L  \partial_t  \kappa \,,
\end{eqnarray*}
extend into:
\begin{multline*}
\int_0^T  \int_{0}^L   \partial_x \left[ \mu \partial_x u - p  \right]  v 
= - \dfrac{1}{2} \left[ \int_0^L \mu \left| \partial_x u - \kappa \right|^2 \right]_0^T + \dfrac{1}{2} \int_0^T \int_0^L \partial_t \mu \left| \partial_x u - \kappa \right|^2 \\- \int_0^T \mathbb E[{\mu \partial_x u - p}]  \int_0^L  \partial_t  \kappa .
\end{multline*}
This simplifies the RHS of \eqref{eq_vf1}, whereas, on the left-hand side, we have:
\begin{eqnarray*}
LHS &=& \int_0^T\int_0^L \rho \left( \partial_t u + u \partial_x u\right) \left( \partial_t u - \partial_x^{-1} [\partial_t \kappa - \mathbb E [{\partial_t \kappa}]]\right) \\
        &\geq & \dfrac{1}{2}\int_0^T \int_0^L \rho \left| \partial_t u + u \partial_x u \right|^2 -   \int_0^T\int_0^L \rho |u \partial_x u|^2 - \int_0^T\int_0^L \rho \left| \partial_x^{-1} [\partial_t \kappa - \mathbb E{[\partial_t \kappa]}]\right|^2.
\end{eqnarray*}
Finally, \eqref{eq_vf1} reduces to\footnote{Note that $\rho (\partial_t u + u \partial_x u ) = \partial_x[\mu \partial_x u - p]$ and $\rho \leq 2 \overline{\rho^0}$}:
\begin{multline} \label{eq_vf2}
\dfrac{1}{2} \left[ \int_0^L \mu \left| \partial_x u - \kappa \right|^2 \right]_{t=T}   + \dfrac{1}{4\overline{\rho^0}} \int_0^T \int_0^L  \left| \partial_x \left[ \mu \partial_x u - p \right] \right|^2
\\
\begin{array}{l}
\displaystyle \leq \dfrac{1}{2} \left[ \int_0^L \mu(\rho^0) \left| \partial_x u^0 - \kappa(\rho^0) \right|^2 \right] +   \dfrac 12 \int_0^T\int_0^L \partial_t \mu \left| \partial_x u - \kappa \right|^2  - \int_0^T \mathbb E[z] \int_0^L  \partial_t  \kappa  \\[14pt]
\displaystyle \qquad +  \int_0^T\int_0^L \rho |u \partial_x u|^2+  \int_0^T \int_0^L \rho \left| \partial_x^{-1} [\partial_t \kappa - \mathbb E[{\partial_t \kappa}]]\right|^2\,, \\[14pt]
\displaystyle  \leq \dfrac{1}{2} \left[ \int_0^L \mu(\rho^0) \left| \partial_x u^0 - \kappa(\rho^0) \right|^2 \right] + \dfrac 12 I_1 - I_2 + I_3 + I_4.
\end{array}
\end{multline}
We bound now $I_1,I_2,I_3,I_4$. 

Applying \eqref{eq_young} with $\beta = 1/\mu,$ we have first:
\begin{eqnarray*}
I_1 &=&  - \int_0^T\int_0^L \partial_t \left[ \frac 1\mu\right] \left| \mu \partial_x u - p \right|^2 \\
	&=&   \int_0^T\int_0^L \partial_x\left[\dfrac u\mu \right]  \left| \mu \partial_x u - p \right|^2   - \int_0^T \int_0^L  \dfrac{(\mu' \rho + \mu)}{\mu^2}\partial_x u  \left| \mu \partial_x u - p \right|^2 \\
	&=&  - 2\int_0^T \int_0^L \dfrac{u}{\mu}[\mu \partial_x u - p] \partial_x[ \mu \partial_x u - p ] -   \int_0^T\int_0^L  \dfrac{(\mu' \rho + \mu)}{\mu^2}\partial_x u  \left| \mu \partial_x u - p \right|^2.
\end{eqnarray*}
Recalling  that  thanks to \eqref{eq_ec0}:
$$
\|u \|_{L^2_{\sharp}}^2  = \int_0^L |u|^2 \leq \dfrac{4}{\underline{\rho^0}} \int_0^L \dfrac{\rho |u|^2}{2} \leq  \dfrac{4\mathcal E_c^0}{\underline{\rho^0}},
$$
we obtain that, for arbitrary small $\varepsilon$:
\begin{eqnarray*}
|I_1| &\leq& {C_0} \int_0^T \left[ \|u \|_{L^{\infty}_{\sharp}} \|z\|_{L^2_{\sharp}} 
                       \|\partial_x z\|_{L^2_{\sharp}}  +  \|\partial_xu\|_{L^{\infty}_{\sharp}} \| z \|_{ L^2_{\sharp}}^2  \right] \\[8pt]
	 &\leq & \dfrac{1}{8\underline{\rho^0}}\int_0^T  \|\partial_x z \|_{L^2_{\sharp}}^2 + C_0 \int_0^T \|z \|_{L^2_{\sharp}}^2\left[  \|u\|_{L^2_{\sharp}}^2 +  \|\partial_xu\|_{ L^{2}_{\sharp}}^2 +  \|\partial_xu\|_{ L^{\infty}_{\sharp}}  \right]    
\end{eqnarray*}
Rewriting $z$ in  terms of $\partial_x u$ and $p(\rho),\mu(\rho),$ we obtain finally that
\begin{equation} \label{eq_I1}
|I_1| \leq {C_0} \int_0^T  \left( \|\partial_x u \|_{ L^2_{\sharp}}^2 +  \|\partial_xu \|_{L^{\infty}_{\sharp}} + 1\right) \int_0^L  \mu |\partial_x u - \kappa|^2  
   + \dfrac{1}{8\underline{\rho^0}} \int_0^T \|\partial_x z \|_{L^2_{\sharp}}^2.
\end{equation}

\medskip

Concerning $I_2 = \int_0^T \int_0^{L} \mathbb E[z] \partial_t \kappa$, we have, applying \eqref{eq_young}:
$$
\partial_t \kappa + \partial_x (\kappa u ) +( \kappa' \rho - \kappa  ) \partial_x u = 0, 
$$
so that
$$
\int_{0}^L \partial_t \kappa = - \int_0^L ( \kappa' \rho - \kappa  ) \partial_x u  
$$
and consequently, with the same arguments as above:
\begin{equation} \label{eq_I2}
|I_2| \leq  \int_0^TC_0 \|z\|_{ L^2_{\sharp}}  \|\partial_x u\|_{L^2_{\sharp}} \leq C_0 \int_0^T  \|\partial_x u \|^2_{ L^2_{\sharp}}  +  C_0 \int_0^T\int_0^L  \mu |\partial_x u - \kappa|^2    .
\end{equation}
Concerning $I_3 = \int_0^T\int_0^L \rho |u\partial_x u|^2,$ we proceed as previously:
\begin{eqnarray*}
|I_3| &\leq& C_0  \int_0^T \|u\|_{L^{\infty}_{\sharp}}^2  \int_{0}^{L}  |\partial_xu|^2\, \\
&\leq &  C_0 \int_0^T \left( 1  +    \|\partial_x u \|_{L^2_{\sharp}}^2\right)  \int_{0}^{L} |\partial_xu|^2\,\\
\end{eqnarray*}
Finally, expressing $\partial_x u$ in terms of $z$ and functions of $\rho,$  there still exists a constant $C_0$ for which:
\begin{equation} \label{eq_I3}
|I_3| \leq C_0   \int_0^T  (1+ \|\partial_x u\|_{L^2_{\sharp}}^2) \int_0^L \mu |\partial_x u - \kappa |^2 +  C_0\int_0^T  (1+ \|\partial_x u\|_{L^2_{\sharp}}^2).
\end{equation}
Then, for $I_4,$ we note as previously that:
$$
\partial_t \kappa - \mathbb E[\partial_t \kappa] = - \partial_x(\kappa u) - [ (\kappa' \rho - \kappa ) \partial_x u - \mathbb E[(\kappa' \rho - \kappa ) \partial_x u ]]
$$
Consequently, there holds:
$$
\partial_x^{-1}   \left[\partial_t \kappa -\mathbb E[\partial_t \kappa]  \right] = - \left[ \kappa u - \mathbb E[\kappa u]\right] - w
$$
where
$$
w = \partial_x^{-1} \left[(\kappa' \rho - \kappa ) \partial_x u - \mathbb E[ (\kappa' \rho - \kappa ) \partial_x u ] \right].
$$
A classical Poincar\'e--Wirtinger inequality  yields that:
$$
\| \partial_x^{-1}   \left[\partial_t \kappa - \mathbb E[{\partial_t \kappa}] \right] \|^2_{L^2_{\sharp}}
\leq  C_0 [ \|u\|^2_{ L^2_{\sharp}} +  \| \partial_x u  
\|^2_{L^2_{\sharp}} ]
$$
Hence $I_4 =\int_0^T\rho |\partial_x^{-1}   \left[\partial_t \kappa -  \mathbb E[{\partial_t \kappa}] \right] |^2$ satisfies:
\begin{equation} \label{eq_I4}
|I_4| \leq \int_0^T C_0 \left(1+  \| \partial_x u\|^2_{ L^2_{\sharp}}\right).
\end{equation}

Combining the computations \eqref{eq_I1}--\eqref{eq_I4} of $I_1,I_2,I_3,I_4,$
 we obtain finally that \eqref{eq_vf2} reads:
\begin{multline} \label{eq_vf3}
 \left[ \int_0^L \mu|\partial_x u - \kappa |^2 \right]_{t=T}   + \dfrac{1}{8 \underline{\rho^0}} \int_0^T \int_0^L  \left| \partial_x z \right|^2
 \\
 \leq    \left[ \int_0^L \mu(\rho^0) \left| \partial_x u^0 - \kappa(\rho^0) \right|^2 \right] +   \int_0^T  f(t)   \int_0^L \mu|\partial_x u - \kappa |^2  + \int_0^T g(t)
\end{multline} 
where
$$
 f(t) = C_0 \left(1+ \|\partial_x u \|_{L^2_{\sharp}}^2 
    +   \|\partial_xu \|_{L^{\infty}(0,L)}  \right).
$$
and 
\begin{eqnarray*}
g(t) &=&C_0 \left( 1+ \| \partial_x u \|_{L^2_{\sharp}}^2 \right) 
\end{eqnarray*}
On the one hand, we have that (we may assume $T < 1$ without restriction so that \eqref{eq_Regd} holds true):
\begin{eqnarray*}
\int_0^T f(t) {\rm d}t &\leq& C_0  \left( T +  \int_0^T   \|\partial_xu \|_{L^{\infty}(0,L)}
+  \|\partial_xu \|^2_{L^2_{\sharp}}\right) \\
						&\leq & C_0 \left( T (1+ K_u^0) + \sqrt{T} K_d^0\right).  
\end{eqnarray*}
Consequently, there exists $T_0^u < 1$ depending only on   $\underline{\rho^0},\overline{\rho^0},\|u^0\|_{H^1_{\sharp}}$  such that:
$$
\exp\left(\int_0^{T_0^u} f(t) {\rm d}t\right) \leq 2.
$$ 
Similarly we have:
\begin{eqnarray*}
\int_0^Tg(t){\rm dt} &\leq & C_0 T \left( 1 + K_u^0  \right)
\end{eqnarray*}
Hence, restricting the size of $T_0^u$ if necessary, but keeping the same dependencies, we have that, for $T < T_0^u:$ 
$$
\int_0^Tg(t){\rm dt}  \leq   \dfrac{1}{2} \left[ \int_0^L \mu(\rho^0) \left| \partial_x u^0 - \kappa(\rho^0) \right|^2  +1\right]. 
$$
Finally, by a standard application of the Gronwall lemma, we obtain then that, for arbitrary $T < T_0^u,$ there holds:
\begin{eqnarray*}
\sup_{t \in [0,T]} \|\sqrt{\mu}(\partial_x u - \kappa)\|^2_{L^2_{\sharp}} 
&\leq&  \displaystyle \exp\left( \int_0^T f(t){\rm dt}\right) \left(  \|\sqrt{\mu(\rho^0)} (\partial_x u _0- \kappa(\rho^0))\|^2_{L^2_{\sharp}}
+ 2 \int_0^T g(s){\rm d}s\right)  \\
&\leq & 4 \displaystyle  \left(\|\sqrt{\mu(\rho^0)} \partial_x u _0- \kappa(\rho^0)\|^2_{ L^2_{\sharp}} + 1 \right)
\end{eqnarray*}
Consequently:
$$
\sup_{t \in [0,T]} \| \partial_x u  \|^2_{L^2_{\sharp}}  \leq \dfrac{8}{\mu^0} \left[ \|\sqrt{\mu(\rho^0)} \partial_x u _0- \kappa(\rho^0) \|^2_{L^2_{\sharp}} + 1 + L|K^0_{p}|^2\right],
$$
and we also have: 
\begin{eqnarray*}
\dfrac{1}{4 \overline{\rho^0}}\int_0^T |\partial_x z |^2{\rm d}s &\leq 
&  \dfrac{1}{2} \left[ \int_0^L \mu(\rho^0) \left| \partial_x u^0 - \kappa(\rho^0) \right|^2 \right]  \\
&& \quad +  4 \int_0^T f(t){\rm dt}   \left( \|\sqrt{\mu(\rho^0)} ( \partial_x u _0- \kappa(\rho^0) )\|^2_{L^2_{\sharp}}  + 1\right)   +  \int_0^T g(s){\rm d}s \\
&\leq & 9  \left(\|\sqrt{\mu(\rho^0)} (\partial_x u _0- \kappa(\rho^0) )\|^2_{L^2_{\sharp}} +1 \right)
\end{eqnarray*}
Finally, we have indeed, that, for arbitrary $T  \in [0,T_0^u)$ there holds:
\begin{multline*}
\sup_{t \in [0,T]} \| \partial_x u\|^2_{ L^2_{\sharp}}  + \int_0^T |\partial_x z |^2{\rm d}s   \\
\leq  \left( \dfrac{8}{\mu^0} + 36 \overline{\rho^0} \right)   \left[ \|\sqrt{\mu(\rho^0)} \partial_x u _0- \kappa(\rho^0) \|^2_{L^2_{\sharp}} + 1 + L|K^0_{\kappa}|^2\right].
\end{multline*}
\end{proof}
Combining  {\bf Lemma \ref{lem_rho}} and {\bf Lemma \ref{lem_uniformbound}}, we obtain finally, that, for $T_0 = \min (1,T_0^{\rho},T_0^{u})/2$
we have $(\textbf{HDS})_c$ with 
$$
K_0  = 36 \left( \dfrac{1}{\mu^0} +  \overline{\rho^0} \right)   \left[ \|\sqrt{\mu(\rho^0)} \partial_x u _0- \kappa(\rho^0) \|_{L^2_{\sharp}}^2 + 1 + L|K^0_{\kappa}|^2\right].
$$

\subsection{Compactness argument : proof of main theorems}

In this last section, we complete the proof of our mains results: {\bf Theorem \ref{thm_existence1}} and {\bf Theorem \ref{thm_multi}}. 
We first remark that the proof of both results reduces to a study of compactness of HD solutions to \eqref{eq_NSC}-\eqref{eq_ee}.

To complete the proof of {\bf Theorem \ref{thm_existence1}}, we remark that, given an initial data $(\rho^0,u^0) \in L^{\infty}_{\sharp} \times H^1_{\sharp}$ we may approximate 
this initial data by a sequence $(\rho^0_n,u^0_n) \in L^{\infty}_{\sharp} \cap H^1_{\sharp}$ satisfying 
\begin{equation} \label{eq_bound1}
\underline{\rho^0} \leq \rho^0_n \leq \overline{\rho^0} \qquad
    \|u^0_n \|_{H^1_{\sharp}}\leq \|u^0 \|_{H^1_{\sharp}}\,, \quad \forall \, n \in \mathbb N\,,
\end{equation}
and 
\begin{equation} \label{eq_conv1}
\rho^0_n \to \rho^0 \text{ in $L^1_{\sharp}$} \qquad u^0_n\to u^0 \text{ in $H^1_{\sharp}$}.  
\end{equation}
This can be done by a standard mollifying/projection argument. Then, the result in the previous section shows that there exists $T_0 >0$
independant of $n \in \mathbb N$ for which there exists a HD solution $(\rho_n,u_n)$   to \eqref{eq_NSC}-\eqref{eq_ee} on $(0,T_0)$ associated
with initial data  $(\rho^0_n,u^0_n).$ Our objective is to prove that we can extract a subsequence of these HD solutions that converges
to an HD solution to  \eqref{eq_NSC}-\eqref{eq_ee} on $(0,T_0)$ associated
with initial data  $(\rho^0,u^0).$

On the other hand, to complete the proof of {\bf Theorem \ref{thm_multi}}, we consider a sequence of initial data $(\rho^0_n,u^0_n) \in L^{\infty}_{\sharp} \cap H^1_{\sharp}$.
Under the assumptions of {\bf Theorem \ref{thm_multi}}, there exists a constant $C_0 \in (0,\infty)$ for which:
\begin{equation} \label{eq_bound2}
\dfrac{1}{C_0} \leq \rho_n^0 \leq C_0 \qquad \|u^0_n\|_{H^1_{\sharp}} \leq C_0 \,, \quad \forall \, n \in \mathbb N.
\end{equation}
Assuming that {\bf Theorem \ref{thm_existence1}} holds (that would result from a first application of the proof below), there exists  $T_0 >0$
independant of $n \in \mathbb N$ for which there exists a HD solution $(\rho_n,u_n)$   to \eqref{eq_NSC}-\eqref{eq_ee} on $(0,T_0)$ associated
with initial data  $(\rho^0_n,u^0_n).$ We aim then at studying if this sequence admits cluster point and to compute a system satisfied by these cluster points.

We first make precise the convergence of the initial data that we apply in  {\bf Theorem \ref{thm_multi}}. We have the definition below:
\begin{definition}
Given $(L,C_0)>0$ we call $L$-periodic Young-measure on $\mathbb R \times [0,2C_0]$ any positive bounded measure $\mu$  on $R \times [0,2C_0],$ $L$-periodic in the first variable, and satisfying:
\begin{equation} \label{eq_youngcar}
\langle \mu , (x,\xi) \mapsto \phi(x) \rangle = \int_0^L \phi(x) {\rm d}x.
\end{equation}
 We denote $\mathcal Y_{\sharp}([0,L] \times [0,2C_0])$ the set of $L$-periodic Young measures. 
\end{definition}
This definition is an adaptation to the periodic framework of the definition of {\sc L. Tartar} \cite{Tartar}. It goes with several remarks:
\begin{enumerate}
\item As $L,C_0$ will be fixed in what follows, we drop it in the notations for Young measures. From now on, we denote simply $\mathcal Y_{\sharp}.$
\item As in the non-periodic case our Young measures form a closed subspace of the set of positive measures on $\mathbb R \times [0,2C_0].$ As we work in an unbounded domain (in $x$), "weak$-*$ convergence" is understood locally (see \cite[Section 2]{Tartar} for more details). 
\item there holds $\mathcal Y_{\sharp} \subset [L^{1}((0,2L);C([0,2C_0]))]^{*}.$ Indeed, if $\phi \in C_c((0,2L) \times [0,2C_0]) ,$ we have then that $|\phi(x,\xi)| \leq \|\phi(x,\cdot) \|_{C([0,2C_0])}$ and, by the positivity of $\nu$ and \eqref{eq_youngcar}:
$$
|\langle \nu , \phi \rangle | \leq \int_{0}^{2L} \|\phi(x,\cdot) \|_{C([0,2C_0])}{\rm d}x.
$$
The embedding property yields then from the fact that $C_c((0,2L) \times [0,2C_0])$ is dense in $L^{1}((0,2L);C([0,2C_0])),$
\item given $\rho \in L^{\infty}_{\sharp}$ such that $\|\rho\|_{L^{\infty}_{\sharp}} \leq 2C_0,$ we define $\nu_{\rho} \in \mathcal Y_{\sharp}$ by
$$
\langle \nu_{\rho}, \beta \rangle = \int_{0}^L \beta(x,\rho(x)) {\rm d}x\,, \quad \forall \, \beta \in C_c(\mathbb R \times [0,2C_0]).
$$ 
\end{enumerate}

\medskip

In the frame of {\bf Theorem \ref{thm_multi}} we assume that there exists $k \in \mathbb N$ and $(\alpha^0_i,\rho^0_i) \in [L^{\infty}_{\sharp}]^{2k}$
for which there holds:
\begin{itemize}
\item $\alpha^0_i \geq 0$ a.e. for $1 \leq i \leq k$ with $\sum_{i=1}^k \alpha_i^0 = 1,$ a.e.,\\
\item $1/C_0 \leq \rho_i^0 \leq C^0$ a.e. for $1 \leq i \leq k$, \\
\item $\nu_{\rho^0_n} \rightharpoonup \nu^0 = \sum_{i=1}^k \alpha_i \delta_{\xi=\rho_i}.$
\end{itemize}
Given the topology on measures, the last item is equivalent to:
$$
\beta(\rho_n) \rightharpoonup \sum_{i=1}^{k} \alpha_i \beta(\rho^i) \text{ in $L^{\infty}_{\sharp}-{w*},$} \quad  \forall \, \beta \in C([0,2C_0]).
$$
We recall in particular, that, if $\|\rho_n^0\|_{L^{\infty}_{\sharp}} \leq C_0$ for arbitrary $n \in \mathbb N$ and $\rho^0_n \to \rho^0$ in $L^1_{\sharp},$ there holds:
$$
\nu_{\rho_n} \rightharpoonup \nu_{\rho} \quad \text{in}  \quad \mathcal Y_{\sharp}-w*.
$$ 
Hence, the compactness study leading to {\bf Theorem \ref{thm_existence1}} is a particular case of the proof of {\bf Theorem \ref{thm_multi}} (in the case $k=1$).
We complete thus the study by the proof of {\bf Theorem \ref{thm_multi}} only.

So, we have now a uniform time $T_0 > 0,$  and a sequence $(\rho_n,u_n)$ of HD solutions associated with data $(\rho_n^0,u_n^0)$ satisfying
\eqref{eq_bound2}.  Thanks to these uniform bounds ${\textbf{(HDS)}_{c}}$ yields:
\begin{enumerate}[$\bullet$]
\item $\rho_n$ is bounded in $L^{\infty}(0,T_0 ; L^{\infty}_{\sharp})$ (from above and by below), and so do $\mu_n:= \mu(\rho_n),\, p_n:= p(\rho_n)$ and $\kappa_n := \kappa(\rho_n),$
\item $u_n$ is bounded in $L^{\infty}(0,T_0 ; H^1_{\sharp}),$
\item $z_n := \mu_n \partial_x u_n - p_n$ is bounded in $L^2(0,T_0 ; H^1_{\sharp}).$
\end{enumerate}
{\bf Lemma \ref{lem_divu}} together with standard computations imply then that
\begin{enumerate}[$\bullet$]
\item $\partial_x u_n$ is bounded in $L^{1}(0,T_0 ; L^{\infty}_{\sharp}) \cap L^{\infty}(0,T_0 ;  L^2_{\sharp}).$
\end{enumerate}
We have thus, up to the extraction of  a subsequence  (that we do not relabel for conciseness):
\begin{enumerate}[$\bullet$]
\item $\rho_n \rightharpoonup \rho, \, p_n \rightharpoonup p^{\infty}, \, \mu_n \rightharpoonup \mu^{\infty}$ and $\kappa_n  \rightharpoonup \kappa^{\infty}$ in $L^{\infty}(0,T_0 ; L^{\infty}_{\sharp})-w*,$
\item $u_n \rightharpoonup u$ in $L^{\infty}(0,T_0 ; H^1_{\sharp})-w*$ with $\partial_x u \in L^1(0,T_0; L^{\infty}_{\sharp})$
\item $z_n \rightharpoonup z^{\infty}$ in $L^2(0,T_0 ; H^1_{\sharp})-w.$
\end{enumerate}
Furthermore, introducing:
$$
\underline{\rho}_{\infty} := \liminf  \left(\inf_{(0,T_0) \times \mathbb R} \rho_{n}\right)_{n\in \mathbb N},\qquad \overline{\rho}_{\infty} := \limsup \left(\sup_{(0,T_0) \times \mathbb R} \rho_n\right)_{n\in\mathbb N} \,, 
$$ 
classical weak convergence arguments also yield that,
\begin{enumerate}[$\bullet$]
\item for a.e. $(t,x) \in (0,T_0) \times (\mathbb R/L\mathbb Z) $ there holds:
\begin{equation} \label{eq_controlrho}
\underline{\rho}_{\infty} \leq \rho \leq  \overline{\rho}_{\infty}
\end{equation}
\item there exists a constant $K_0$ depending only on $C_0$ and $\sup_{n\in \mathbb N}\|u^0_n\|_{H^1_{\sharp}}$ for which
\begin{equation} \label{eq_controlH1}
\sup_{t \in (0,T_0)} \|u(t,\cdot)\|_{H^1_{\sharp}} 
+ \int_0^{T_{0}} \|z^{\infty}(t,\cdot) \|_{H^1_{\sharp}} \leq K_0.
\end{equation}
\end{enumerate}

\subsubsection{Convergence of momentum equation}
We want now to pass to the limit in the momentum equation satisfied by $\rho_n$ and $u_n.$
To this end, we first obtain strong-compactness for two quantities. We have:
\begin{lemma}
Up to the extraction of a subsequence, we have that 
$$
u_n \to u \text{ in $C([0,T_0] ; L^2_{\sharp})$}.
$$
\end{lemma}
\begin{proof}
We already have that $u_n$ is bounded in $C([0,T_0] ; L^2_{\sharp}) \cap L^{\infty}(0,T_0;H^1_{\sharp})$
where $H^1_{\sharp} \subset L^2_{\sharp}$ is compact.  
Furthermore, we have from the continuity equation that:
$$
\partial_t u_n =  - u_n \partial_x u_n   + \dfrac{1}{\rho_n} \partial_x z_n.
$$
Consequently: 
\begin{multline*}
\|\partial_t u_n \|_{L^2(0,T_0;L^2_{\sharp})} \leq
 \|u_n \|_{L^{\infty}(0,T_0 ; L^{\infty}_{\sharp})}
  \|u_n \|_{L^{2}(0,T_0 ; H^{1}_{\sharp})}
+ \||\rho_n|^{-1}  \|_{L^{\infty}(0,T_0 ; L^{\infty}_{\sharp})} 
    \|\partial_x z_n \|_{L^{2}(0,T_0 ; L^2_{\sharp})}.
\end{multline*}
But, the bounds claimed above and the embedding $H^1_{\sharp} \subset L^{\infty}_{\sharp}$ yield that the right-hand side of this inequality is bounded uniformly in $n \in \mathbb N.$ Consequently, we have that $u_n$
is uniformly equicontinuous in $C([0,T_0] ; L^2_{\sharp})$ and we may extract a strongly converging subsequence.
\end{proof}

{\bf Remark.}  We can then prove that $\rho^n |u^n|^2  \to \rho |u|^2$ (in $L^2((0,T) \times (\mathbb R/L\mathbb Z)) -w$ 
for instance)  and,  if $\rho_n^0$ converges strongly to $\rho^0$, classical arguments on the dissipation estimate satisfied by $(\rho^n,u^n)$ imply that for a.e. $t \in (0,T_0)$ there holds:
\begin{equation} \label{eq_controlL2}
\int_{0}^L \left[ \dfrac {\rho(t,\cdot) |u(t,\cdot)|^2}{2} +  q^{\infty} \right]    + \int_0^t \int_0^L \mu |\partial_x u|^2   \leq  \int_{0}^L \left[ \dfrac{\rho^0|u^0|^2}{2} +  q^0 \right]  
\end{equation}
where $q^0 = q(\rho^0).$

Second, we  state the equivalent result to the viscous-flux lemma that was crucial to the proof by P.-L. {\sc Lions} \cite{Li2} and by E. {\sc Feireisl}, A. {\sc Novotn\'y} and H. {\sc Petzeltov\'a}  \cite{FeNoPe} to obtain
existence of global weak solutions to compressible Navier--Stokes systems:

\begin{lemma} \label{lem_compcomp}
Let $\beta \in C^{1}((0,\infty))$ then, up to the extraction of a subsequence, we have that
\begin{eqnarray*}
\beta(\rho_n) \rightharpoonup \beta^{\infty} & & \text{ in } L^{\infty}(0,T_0 ; L^{\infty}_{\sharp})-w*\,, \\
\beta(\rho_n) z_n \rightharpoonup \beta^{\infty} z^\infty && \text{ in $L^2((0,T_0) \times \mathbb R/L\mathbb Z)-w$}.
\end{eqnarray*}
\end{lemma}
\begin{proof}
Under the assumptions of this lemma (and keeping the conventions of the previous section for the operator $\partial_x^{-1}$), we set:
$$
\beta_n = \beta(\rho_n)\,, \quad w_n = \partial_x^{-1} \left[ \beta_n - \mathbb E[{\beta_n}]\right].
$$ 
Then, $\beta_n$ and $w_n$ are bounded respectively in $C([0,T_0];L^2_{\sharp}) \cap L^{\infty}((0,T_0) \times \mathbb R/L\mathbb Z )$  and $C([0,T_0] ; H^1_{\sharp}).$ In particular, we may extract a subsequence s.t. $\beta(\rho^n)$ and  $\beta(\rho_n) z_n$ converge respectively in  $ L^{\infty}((0,T_0) \times \mathbb R/L\mathbb Z )-w*$ and  $L^2((0,T_0) \times \mathbb R/L\mathbb Z)-w.$
We denote $\beta^{\infty}$ the weak$-*$ limit of $\beta(\rho_n).$ Furthermore, there holds: 
$$
\partial_t w_n = \partial_x^{-1} \left[ \partial_t \beta_n - \mathbb E[{\partial_t \beta_n}]\right] 
$$
where, as previously:
\begin{eqnarray*}
\partial_t \beta_n &=& - \partial_x (\beta_n u_n) - (\beta'(\rho_n) \rho_n - \beta_n) \partial_x u_n \in L^{\infty}(0,T_0 ; H^{-1}_{\sharp}) \\
\partial_t \mathbb E[\beta_n] &=& - \mathbb E [ (\beta'(\rho_n) \rho_n - \beta_n) \partial_x u_n ] \in L^{\infty}(0,T_0).
\end{eqnarray*}
Consequently :
$$
\partial_t w_n = - \left( \beta_n u_n - \mathbb E[{\beta_n u_n}] \right) - \partial_x^{-1}    \left[(\beta'(\rho_n) \rho_n - \beta_n) \partial_x u_n  - \mathbb E[{(\beta'(\rho_n) \rho_n - \beta_n) \partial_x u_n}]  \right]
$$
and repeating the computations for $I_4$ in the previous paragraph, we obtain that $\partial_t w_n$ is also bounded in $L^{\infty}(0,T_0;L^2_{\sharp}).$
We may again extract a subsequence we do not relabel such that:
$$
w_n \rightarrow w^{\infty} =\partial_x^{-1} \left[ \beta^{\infty} - \mathbb E[{\beta^{\infty}}]\right] \text{ in $C([0,T_0] ; L^2_{\sharp})$}
$$
(as this is the only possible limit). Consequently also $\mathbb E[\beta_n] $ is bounded in $W^{1,\infty}((0,T_0))$ so that we may extract
a subsequence for which $\mathbb E[\beta_n] \to \mathbb E[\beta^{\infty}]$ in $C([0,T_0]).$

\medskip

For any $n\in \mathbb N$ and $\varphi \in C^{\infty}_c((0,T_0) \times (0,L))$ we have then:
\begin{eqnarray*}
\int_0^{T_0} \int_0^L \beta_n z_n \varphi &=& \int_0^{T_0} \int_0^L \partial_x w_n z_n \varphi + \int_0^{T_0} \int_0^{L} \mathbb E[{\beta_n}] z_n \varphi\,, \\
\end{eqnarray*}
On the one-hand, we have:
\begin{eqnarray*}
 \int_0^{T_0} \int_0^{L} \mathbb E[{\beta_n}] z_n \varphi &=& \dfrac{1}{L} \int_0^{T_0} \int_0^L \beta_n  \int_0^L z_n \varphi  \\
 			&\longrightarrow &  \dfrac{1}{L} \int_0^{T_0} \int_0^L \beta^{\infty}  \int_0^L z^\infty \varphi 
\end{eqnarray*} 
due to the strong convergence of $(\mathbb E(\beta_n))_{n\in \mathbb N}$ in $C([0,{T_0}])$ and the weak convergence
of $(\int_0^L z_n\varphi)_{n\in \mathbb N}$ in $L^2((0,{T_0})).$ 

On the other hand, there holds:
$$
\int_0^{T_0} \int_0^L \partial_x w_n z_n \varphi  = -\int_0^{T_0} \int_0^L w_n \varphi \partial_x z_n  -\int_0^{T_0} \int_0^L  w_n z_n \partial_x \varphi  
$$
Combining, the strong convergence of $w_n$ in $C([0,T_0];L^2_{\sharp})$
and the weak convergence of $z_n$  in $L^2(0,T_0 ; H^1_{\sharp})$ we also get:
$$
\int_0^{T_0} \int_0^L  \partial_xw_n  z_n  \varphi   \longrightarrow -\int_0^{T_0} \int_0^L w^{\infty} \varphi \partial_x z^{\infty} -\int_0^{T_0} \int_0^L  w^{\infty} z^{\infty} \partial_x \varphi 
$$
These computations entail finally that
\begin{eqnarray*}
\lim_{n\to \infty} \int_0^{T_0} \int_0^L \beta_n z_n \varphi &=& -\int_0^{T_0} \int_0^L w^{\infty} \varphi \partial_x z^\infty -\int_0^{T_0} \int_0^L  w^{\infty} z^{\infty} \partial_x \varphi + \int_0^{T_0} \int_0^{L} \mathbb E[{\beta^{\infty}}] w^{\infty} \varphi \\
&=& \int_0^{T_0} \int_0^L \beta^{\infty} z^\infty \varphi.
\end{eqnarray*}
This completes the proof.
\end{proof}

We can now pass to the limit in the equations satisfied by $(\rho_n,u_n).$ 

\begin{proposition}
We have in $\mathcal D'((0,T_0) \times \mathbb R):$
$$
\partial_t (\rho u) + \partial_x (\rho u^2) = \partial_x \left[ m^{\infty} (\partial_x u  - \kappa^{\infty})\right].
$$
where 
$$
m^{\infty} = \left[{\lim \dfrac{1}{\mu(\rho_n)}}\right]^{-1}.
$$
\end{proposition}
\begin{proof}
We recall that, for any given $n\in\mathbb N$ there holds:
$$
\partial_t (\rho_n u_n) + \partial_x (\rho_n |u_n|^2) = \partial_x z_n.
$$
Combining the weak convergences of $\rho_n$ and $u_n$ and the strong convergence of $u_n,$ we obtain that 
\begin{enumerate}[$\bullet$]
\item $\rho_n u_n \rightharpoonup \rho u$ in $L^2((0,T_0) \times (\mathbb R/L\mathbb Z))-w,$
\item $u_n \to u$ in $L^4((0,T_0) \times (\mathbb R/L\mathbb Z) )$ or $|u_n|^2 \to u^2$ in $L^2((0,T_0) \times (\mathbb R/L\mathbb Z) )$,
\item $\rho_n |u_n|^2 \to \rho u^2 $ in $L^2((0,T_0) \times (\mathbb R/L\mathbb Z))-w.$
\end{enumerate}
This enables to pass to the limit in the right-hand side:
$$
\partial_t (\rho_n u_n) + \partial_x (\rho_n |u_n|^2) \rightharpoonup \partial_t (\rho u) + \partial_x (\rho u^2) \text{ in $\mathcal D'((0,T_0) \times \mathbb R)$} 
$$
On the right-hand side we have that $z_n \rightharpoonup z^\infty$ so that:
$$
\partial_t (\rho u) + \partial_x (\rho u^2) = \partial_x z^\infty.
$$
It remains to compute $z$ in terms of $\rho$ and $u.$ We have, for fixed $n\in \mathbb N:$
$$
\partial_x u_n = \dfrac{z_n}{\mu_n} + \kappa(\rho_n). 
$$
Passing to the limit in this identity (in $L^2((0,T_0) \times (\mathbb R/L\mathbb Z))-w$ for instance), we get, thanks to the previous lemma:
$$
\partial_x u = \lim \left[ \dfrac{1}{\mu(\rho_n)}\right] z^{\infty}+ \kappa^{\infty} \text{ or } z^{\infty} = m^{\infty} \left( \partial_x u - \kappa^{\infty}\right).
$$
This ends the proof of this proposition.
\end{proof}

As classical in these compactness arguments, the main difficulty now is to find a relation between 
$\mu^{\infty},$  $\kappa^{\infty}$ and $\rho.$ In full generality, this is not possible:
the operators "$\lim$" and  the operator "composition by a continuous function $\beta$" do not commute.
To analyze more precisely the commutators, we apply Young-measure theory.

\subsubsection{Compactness of Young measures.}
For a given $n \in \mathbb N$ we introduce the young measure $\nu_n := \nu_{\rho_n}.$
From the regularity  $\rho_n \in C([0,T_0] ; L^1_{\sharp})$ we deduce that  $\nu_n \in C([0,T_0] ; \mathcal Y_{\sharp}-w*).$
We state then
\begin{proposition}
There exists a subsequence we do not relabel such that  $\nu_n \rightharpoonup \nu$  in $C([0,T_0] ; \mathcal Y_{\sharp}-w*)$
Furthermore, $\nu$ is a solution to:
\begin{equation} \label{eq_young3}
\partial_t \nu + \partial_x (\nu u) -  \left( \partial_{\xi}\left(\xi \dfrac{\nu}{\mu(\xi)} \right) + \dfrac{\nu}{\mu(\xi)} \right) z^{\infty}  - \left(\partial_{\xi}\left(\frac{\xi p(\xi)\nu}{\mu(\xi)}\right) + \frac{\nu p(\xi)}{\mu(\xi)}\right)  = 0
\end{equation}
in $\mathcal D'((0,T_0) \times \mathbb R \times (0,2C_0))$ with initial condition:
\begin{equation} \label{eq_young3bis}
\nu(0,\cdot) = \sum_{i=1}^k \alpha_i^0 \delta_{\rho_i^0} .
\end{equation}
\end{proposition}
\begin{proof}
By construction, $\nu_n \in C([0,T_0] ; \mathcal Y_{\sharp}-w*)$ and is a bounded sequence in this space.
Then,  we rewrite \eqref{eq_young}:
\begin{equation} \label{eq_young2} 
\partial_t \nu_n + \partial_x (\nu_n u_n) -  (\partial_{\xi}(\xi \nu_n) + \nu_n) \partial_x u_n = 0 \quad \text{ in $\mathcal D'((0,T_0) \times \mathbb R \times (0,2C_0))$}. 
\end{equation}
For arbitrary  $\phi \in C^{\infty}_c(\mathbb R \times (0,2C_0))$ there holds:
$$
\partial_t \langle \nu_n , \phi \rangle =  \langle \nu_n , u_n \partial_x \phi \rangle + \langle \nu_n ,  \phi \partial_x u_n \rangle - \langle \nu_n , \xi  \partial_x u_n \partial_{\xi}\phi\rangle.
$$
We recall here that $u_n$ is bounded in $L^{\infty}((0,T); C_{\sharp})$ and that $\partial_x u_n$ is bounded also in the space $L^{1}(0,T ; L^{\infty}_{\sharp}).$  Hence, for arbitrary $\phi \in  C^{\infty}_c( \mathbb R \times (0,2C_0))$ we have that  $(\langle \nu_n , \phi \rangle)_{n\in \mathbb N}$ is relatively compact in  $C([0,T_0]).$

As  the weak$-*$ convergence on $\mathcal Y_{\sharp}$ measures only the weak$-*$ convergences in all the  $C_c((-N,N) \times (0,2C_0))^*$ (for  $N \in  \mathbb N$, since $\nu_n$ has support in $\mathbb R \times [1/C_0,C_0])$)  which  admits a denumerable dense set of functions
belonging to $C^{\infty}_c( (-N,N) \times (0,2C_0)),$ we may apply a classical argument to obtain that, up to the extraction of a subsequence, $\nu_n$ converges in $C([0,T_0] ; \mathcal Y_{\sharp}-w*).$

Then, we rewrite equivalently \eqref{eq_young2} as:
$$
\partial_t \nu_n + \partial_x (\nu_n u_n) -  \left( \partial_{\xi}\left(\xi \dfrac{\nu_n}{\mu(\xi)} \right) + \dfrac{\nu_n}{\mu(\xi)} \right) z_n  - \left(\partial_{\xi}\left(\frac{\xi p(\xi)\nu_n}{\mu(\xi)}\right) + 
  \frac{\nu_n p(\xi)}{\mu(\xi)}\right)  = 0
$$
Combining the weak convergence of $\nu_n,z_n,$ the strong convergence of $u_n$ and {\bf Lemma \ref{lem_compcomp}} we may pass to the limit in this equation and obtain \eqref{eq_young3}.
\end{proof}

To end the proof of {\bf Theorem \ref{thm_existence1}} we remark that \eqref{eq_young3}-\eqref{eq_young3bis} enters the framework of {\bf Appendix \ref{app_young}}. 
Indeed, we rewrite \eqref{eq_young3} as 
$$
\partial_t \nu + \partial_x (\nu u_x) + \partial_{\xi} (\nu u_{\xi}) + g \nu = 0
$$
with, thanks to the previous arguments  $u_x = u(t,x) \in C([0,T_0] \times \mathbb R/L\mathbb Z)$ s.t.:
$$
 \partial_x u_x = \dfrac{z^{\infty}}{m^{\infty}} + \kappa^{\infty} \in L^{1}(0,T_0;L^{\infty}_{\sharp}) \,,  \quad \partial_{\xi} u_{x} = 0\,.
$$
and (note that $z^{\infty} \in L^2(0,T_0; H^1(\mathbb R/L\mathbb Z)) \subset L^1 (0,T_0; C(\mathbb R/L\mathbb Z))$  ):
\begin{align*}
& u_{\xi} = - \left( \dfrac{\xi}{\mu(\xi)} z^{\infty}(t,x) + \frac{\xi p(\xi)}{\mu(\xi)}\right) \in L^{1}(0,T_0 ;C(\mathbb R/L\mathbb Z \times [0,2C_0])) 
\end{align*}
such that:
\begin{align*}
&\partial_{\xi} u_{\xi} \in L^{1}(0,T_0;L^{\infty}(\mathbb R/L\mathbb Z \times (0,2C_0)))  ; \\[4pt]
& \int_0^T \int_0^L \sup_{\xi \in [0,2C_0]} |\partial_x u_{\xi}(t,x,\xi)|{\rm d}\xi {\rm d}x < \infty\,.
\end{align*}
Finally, we have:
\begin{align*}
&g = - \left( \dfrac{z^{\infty}(t,x) }{\mu(\xi)} +  \frac{p(\xi)}{\mu(\xi)}\right)\in L^{1}(0,T_0 ;C(\mathbb R/L\mathbb Z \times [0,2C_0]))  
\end{align*}
such that:
\begin{align*}
&\partial_{\xi} g \in L^{1}(0,T_0;L^{\infty}(\mathbb R/L\mathbb Z \times (0,2C_0)))  , \\[4pt]
& \int_0^T \int_0^L \sup_{\xi \in [0,2C_0]} |\partial_x g(t,x,\xi)|{\rm d}\xi {\rm d}x < \infty\,.
\end{align*}
Hence,  {\bf Appendix \ref{app_young}} ensures that $\nu$ is the unique solution to \eqref{eq_young3}-\eqref{eq_young3bis} and that it writes as a convex combination of $k$ Dirac measures. 
Plugging formally  $\nu = \sum_{i=1}^k \alpha_i \delta_{\xi = \rho_i}$ in \eqref{eq_young3}-\eqref{eq_young3bis} we get that the $(\alpha_i,\rho_i)$ are solutions of the expected pde system. 
We note that these equations are actually satisfied by construction (see the proof of {\bf Lemma \ref{lem_youngcons}}). This ends the proof of {\bf Theorem \ref{thm_multi}}.

Let mention that, in the particular case $k=1,$ we recover that $\alpha_1 = 1$ and that $(\rho^1,u) = (\rho,u)$ satisfies \eqref{eq_NSC}-\eqref{eq_ee}. We also have that \eqref{eq_controlrho}-\eqref{eq_controlL2}-\eqref{eq_controlH1} imply that \eqref{eq_regathm}-\eqref{eq_regbthm}-\eqref{eq_regcthm} holds true on $(0,T_0)$. This completes the proof of 
 {\bf Theorem \ref{thm_existence1}}.

\appendix
 
\section{Complementary result on transport equation} \label{app_young}
In this section we consider periodic  young-measures solution to the transport equation:
\begin{equation} \label{eq_youngapp}
\partial_t \nu +{\rm div}(u\nu) + g \nu = 0,
\end{equation}
in $\mathcal D'((0,T) \times \mathbb R \times (0,M)),$ with initial condition:
\begin{equation} \label{eq_youngappini}
\nu(0,\cdot) = \nu^0.
\end{equation} 
For legibility, we turn to notations $(x_1,x_2)$ for space variables and $u=(u_1,u_2)$ for velocity-fields.
Throughout this appendix, we assume that this velocity-field satisfies:
\begin{itemize}
\item $u_1   \in L^1(0,T ; C(\mathbb R /L\mathbb Z \times [0,M]))$ with:
\begin{align}
& \partial_1 u_1 \in L^1((0,T) ; L^{\infty}((\mathbb R /L\mathbb Z) \times (0,M)));  \label{eq_hypu11}\\
& \partial_{2} u_{1} = 0  \quad \text{ a.e.}\label{eq_hypu12} 
\end{align}
\item $u_{2} \in L^1((0,T) ; C((\mathbb R /L\mathbb Z) \times [0,M]))$ with:
\begin{align}
& \int_0^T \int_{0}^L \sup_{x_2 \in [0,M]} |\partial_1 u_2(t,x_1,x_2)|{\rm d}x_1 {\rm d}t < \infty \label{eq_hypu21}\\[8pt]
&\partial_{2} u_{2} \in L^1((0,T) ; L^{\infty}((\mathbb R /L\mathbb Z) \times (0,M))). \label{eq_hypu22}
\end{align}
\end{itemize}
As for the source term $g$, we assume that
\begin{itemize}
\item $g  \in  L^1(0,T ; C^{0}(\mathbb R/L\mathbb Z \times [0,M]) )$ with
\end{itemize}
\begin{align}
& \int_0^T \int_{0}^L \sup_{x_2 \in [0,M]} |\partial_1 g(t,x_1,x_2)| {\rm d}x_1 {\rm d}t < \infty \label{eq_hypg1}\\[8pt] 
&\partial_{2} g \in L^1((0,T) ; L^{\infty}((\mathbb R /L\mathbb Z) \times (0,M))). \label{eq_hypg2}
\end{align}

We first obtain a uniqueness result:

\begin{lemma}
For arbitrary  $\nu^0 \in \mathcal Y_{\sharp}$  and  $\mathcal K  \Subset \mathbb R/L\mathbb Z \times (0,M),$  
there exists $T_* \leq T$ such that \eqref{eq_youngapp}-\eqref{eq_youngappini} admits at most one solution $\nu \in C([0,T_*] ; \mathcal Y_{\sharp}-w*)$ with support in $\mathcal K.$
\end{lemma}
\begin{proof}
We provide a proof with a duality-regularization argument. By difference, we prove that if $\nu$ satisfies :
\begin{itemize}
\item $\nu$ is a continuous function on $[0,T_*]$  with values in periodic measures on $\mathbb R \times [0,M]$ (endowed with the $w*$-topology) 
\item $\nu_{t}$ has support in $\mathcal K \Subset \mathbb R/L\mathbb Z \times (0,M)$ for arbitrary $t \in (0,T)$ and $\nu_{|_{t=0}} =0$ 
\item for arbitrary $\varphi \in \mathcal D((0,T) \times \mathbb R \times (0,M))$ we have:
$$
\int_0^T \langle \nu , \partial_t \varphi   + u \cdot \nabla \varphi - g \varphi \rangle   = 0\,,
$$
\end{itemize}
then $\nu$ vanishes globally on $[0,T].$ 

First, by a standard regularization argument, we have that, for arbitrary $t \in [0,T]$ and $\varphi \in W^{1,1}([0,t] ; C^1_c(\mathbb R \times (0,M)))$ there holds:
\begin{equation} \label{eq_wfode}
\langle \nu_t , \varphi(t,\cdot) \rangle = -  \int_0^t \langle \nu_s , \partial_t \varphi   + u \cdot \nabla \varphi - g \varphi \rangle {\rm d}s.
\end{equation}
We also fix $\mathcal K'$ containing strictly $\mathcal K$ with $\mathcal K' \Subset  (\mathbb R/L\mathbb Z) \times (0,M)$  and remark that there exists $T_*$ depending
on $u$ and $\mathcal K'$ for which any characteristics $\Gamma$ of the flow associated to $u$ crossing $\mathcal K'$  on $[0,T_*]$ satisfies $\Gamma \Subset  (\mathbb R/L\mathbb Z) \times (0,M).$

Then, we introduce mollified velocities and source term $(u^{\varepsilon},g^{\varepsilon})_{\varepsilon >0}$ obtained by convolution with  tensorized mollifiers $(\rho_{\varepsilon})_{\varepsilon >0}.$
Given the assumed regularity on $u$ and $g$ we have that $(u^{\varepsilon},g^{\varepsilon}) \in L^1(0,T ; C^1((\mathbb R/L\mathbb Z) \times [0,M]) ). $
We shall use the following convergence afterwards:
\begin{itemize}
\item we have the classical convergences 
\begin{equation} \label{eq_convu2}
\|u_2^{\varepsilon} - u_2 \|_{L^1(0,T; L^{\infty})} +  \|g^{\varepsilon} - g\|_{ L^1(0,T; L^{\infty} )} = 0.
\end{equation}
\item thanks to \eqref{eq_hypu11}-\eqref{eq_hypu12} we have  $\nabla u_1 \in L^1(0,T; L^{\infty}(\mathbb R/L\mathbb Z \times (0,M)))$ and:
\begin{equation} \label{eq_convu1}
\|u_1^{\varepsilon} - u_1 \|_{L^1(0,T; L^{\infty} )}  \leq C \varepsilon
\end{equation}
\item applying \eqref{eq_hypu21}-\eqref{eq_hypg1} in the computations of $\partial_2 g^{\varepsilon}$ and $\partial_2 u_2^{\varepsilon}$ and $\nabla u_1$
we obtain the uniform bounds:
\begin{equation} \label{eq_unifug}
\|\partial_2 u_2^{\varepsilon}   \|_{L^1(0,T; L^{\infty})} 
+  \|\partial_2 g^{\varepsilon}  \| _{ L^1(0,T; L^{\infty}) }
 + \|\nabla u_1^{\varepsilon}  \|_{L^1(0,T; L^{\infty})} \leq C.
\end{equation}
\item applying \eqref{eq_hypu21}-\eqref{eq_hypg1} in the computations of $\partial_1 g^{\varepsilon}$ and $\partial_1 u_2^{\varepsilon}$
we obtain the divergences:
\begin{equation} \label{eq_divug}
\|\partial_1 u_2^{\varepsilon}  \|_{L^1(0,T; L^{\infty})}
+  \|\partial_1 g^{\varepsilon} \|_{ L^1(0,T; L^{\infty})}   \leq \dfrac{C}{\sqrt{\varepsilon}}.
\end{equation}
\end{itemize}

Let now $t \in (0,T_*)$ and  $\varphi^{\sharp} \in C^1_c(\mathbb R \times (0,M))$ with support in $\mathcal K'$  we construct now  $\varphi^{\varepsilon}$
solution  to 
\begin{align*}
\partial_t \varphi^{\varepsilon} + u^{\varepsilon} \cdot \nabla \varphi^{\varepsilon}  &=  g^{\varepsilon} \varphi^{\varepsilon} &  \text{ on $(0,t) \times \mathbb R  \times (0,M)\,,$}\\
\varphi(t,\cdot) &= \varphi^{\sharp} & \text{ on $ \mathbb R  \times (0,M)\,.$}
\end{align*}
Classical results on convection equations yield that $\varphi^{\varepsilon}$ has the requested regularity to be a test-function in \eqref{eq_wfode} for $\varepsilon$ sufficiently small. 
In particular, the convergence of the flow associated with $u^{\varepsilon}$ towards the flow associated with $u$ ensures that,  for $\varepsilon$ sufficiently small, $\varphi^{\varepsilon}(s,\cdot)$ 
has compact support in $\mathbb R \times (0,M)$ for any $s \in [0,t].$ Consequently, we have
$$
\langle \nu_t , \varphi^{\sharp} \rangle = -  \int_0^t \langle \nu ,  (u - u^{\varepsilon}) \cdot \nabla \varphi^{\varepsilon} - (g - g^{\varepsilon}) \varphi^{\varepsilon} \rangle {\rm d}s.
$$
This entails that:
$$
|\langle \nu_t , \varphi^{\sharp} \rangle|  \leq  C\left[I_1 + I_2 + I_3\right]  
$$
where :
\begin{eqnarray*}
I_1 &=& \int_0^t \|(u_1 - u_1^{\varepsilon}) \partial_1 \varphi^{\varepsilon} \|_{L^{\infty}(\mathbb R \times (0,M))} \\
I_2 &=&  \int_0^t \|(u_2 - u_2^{\varepsilon}) \partial_2 \varphi^{\varepsilon} \|_{L^{\infty}(\mathbb R  \times (0,M))} \\
I_3 &=& \int_0^t   \|(g - g^{\varepsilon}) \varphi^{\varepsilon}\|_{ L^{\infty}(\mathbb R  \times (0,M))} 
\end{eqnarray*}

Concerning $I_3$ at first, we apply classical maximum-principle arguments  yielding that, for any $s \in (0,t)$ :
$$
\|\varphi^{\varepsilon}(s,\cdot) \|_{L^{\infty}}  \\\leq \|\varphi^{\sharp}\|_{ L^{\infty}} \exp \left(\int_0^t \|g^{\varepsilon} \|_{L^{\infty}} \right). 
$$
Due to the convergence of $g^{\varepsilon}$ towards $g$ we obtain that $\varphi^{\varepsilon}$ is uniformly bounded independently of $\varepsilon$ and  that 
$$
|I_3| \leq C \int_0^t \|g - g^{\varepsilon} \|_{L^{\infty}}  \to 0 \quad \text{ when $\varepsilon \to 0.$} 
$$
Then, we differentiate the transport equation for $\varphi^{\varepsilon}$ w.r.t. $x_2.$ As $\partial_2 u_1^{\varepsilon} = 0,$ we obtain that $\varphi^{\varepsilon}_2 = \partial_2 \varphi^{\varepsilon}$ satisfies:
\begin{align*}
\partial_t \varphi_2^{\varepsilon} + u^{\varepsilon} \cdot \nabla \varphi^{\varepsilon}_2  &=  \partial_2 g^{\varepsilon} \varphi^{\varepsilon} +  g^{\varepsilon} \varphi^{\varepsilon}_2 - \partial_2 u^{\varepsilon}_2 \varphi_2^{\varepsilon}  &  \text{ on $(0,t) \times \mathbb R  \times (0,M)\,,$}\\
\varphi_2^{\varepsilon}(t,\cdot) &= \partial_2 \varphi^{\sharp} & \text{ on $ \mathbb R  \times (0,M)\,.$}
\end{align*}
Refering again to a maximum principle argument for transport equations,  we obtain that, for any 
$s \in (0,t)$ :
$$
\|\varphi_2^{\varepsilon}(s,\cdot) \|_{L^{\infty}}
\leq  \left( \|\partial_2 \varphi^{\sharp}\|_{L^{\infty}} + \int_0^t \|\partial_2 g^{\varepsilon} \varphi^{\varepsilon} \|_{L^{\infty}} \right) \\
\exp \left(\int_0^t \|g^{\varepsilon}  \|_{L^{\infty}} +\|\partial_2 u_2 \|_{L^{\infty}} \right) .
$$
Applying the uniform bound on $\varphi^{\varepsilon}$ together with \eqref{eq_unifug} we get:
$$
\|\varphi_2^{\varepsilon}(s,\cdot) \|_{L^{\infty}} \leq  C\,, \quad \forall \, s \in (0,t).
$$
Combining this remark with the convergence \eqref{eq_convu2} we obtain then:
\begin{eqnarray*}
|I_2| & \leq &  C \sup_{(0,t)} \|\varphi_2^{\varepsilon} \|_{ L^{\infty}} 
         \int_0^t \|(u^{\varepsilon}_2 - u_2) \|_{L^{\infty}} \\
      & \leq &  C \int_0^t \|(u^{\varepsilon}_2 - u_2)\|_{L^{\infty} }  \to 0 \quad   \text{when $\varepsilon \to 0.$}
\end{eqnarray*}
Finally, to compute $I_1$ we differentiate the transport equation for $\varphi^{\varepsilon}$ w.r.t. $x_1.$ We obtain that $\varphi^{\varepsilon}_1 = \partial_1 \varphi^{\varepsilon}$ satisfies:
\begin{align*}
\partial_t \varphi_1^{\varepsilon} + u^{\varepsilon} \cdot \nabla \varphi^{\varepsilon}_1  &=  \partial_1 g^{\varepsilon} \varphi^{\varepsilon} +  g^{\varepsilon} \varphi^{\varepsilon}_1 - \partial_1 u^{\varepsilon}_2 \varphi_2^{\varepsilon} - \partial_1 u_1^{\varepsilon} \varphi_1^{\varepsilon} &  \text{ on $(0,t) \times \mathbb R  \times (0,M)\,,$}\\
\varphi^{\varepsilon}_1(t,\cdot) &= \partial_1\varphi^{\sharp} & \text{ on $ \mathbb R  \times (0,M)\,.$}
\end{align*}
Again, this yields that, for any $s \in (0,t)$ : 
\begin{multline*}
\|\varphi_1^{\varepsilon}(s,\cdot) \|_{L^{\infty}}
\leq  \left( \|\partial_1 \varphi^{\sharp} \|_{ L^{\infty}} + \int_0^t  (\|\partial_1 g^{\varepsilon} \varphi^{\varepsilon}  \|_{L^{\infty}} +  \|\partial_1 u_2^{\varepsilon} \varphi_2^{\varepsilon} \|_{ ;L^{\infty}}  ) \right) \\
\\
\exp \left(\int_0^t \|g^{\varepsilon} \|_{L^{\infty}} +\|\partial_1 u^{\varepsilon}_1 \|_{ L^{\infty}} \right) 
\end{multline*}
Applying the uniform bound on $\varphi^{\varepsilon}$ and $\varphi_2^{\varepsilon}$ with \eqref{eq_unifug} and \eqref{eq_divug} we conclude that
$$
\sup_{s \in (0,t)} \|\varphi_1^{\varepsilon}(s,\cdot) \|_{L^{\infty}}  \leq \dfrac{C}{\sqrt{\varepsilon}}.
$$
Combining this remark with the convergence \eqref{eq_convu1} we obtain then:
\begin{eqnarray*}
|I_1| & \leq &  C \sup_{(0,t)} \|\varphi_1^{\varepsilon} \|_{L^{\infty}(\mathbb R \times (0,M))} \int_0^t \|(u^{\varepsilon}_1 - u_1)\|_{L^{\infty}((\mathbb R /L\mathbb Z) \times (0,M))} \\
      & \leq &  C \sqrt{\varepsilon}  \to 0 \quad \text{ when $\varepsilon \to 0.$}
\end{eqnarray*}
Finally, we have $\langle \nu_t , \varphi^{\sharp} \rangle = 0 $ whatever the value of $\varphi^{\sharp}$. As $\nu_t$ has support in $\mathcal K$ strictly contained in $\mathcal K'$ we conclude that $\nu_t = 0$ globally.
  \end{proof}

We then construct solutions for initial data which are convex combinations of Dirac measures.
Namely, we assume that there exists $(\alpha_i^0,\rho_i^0)_{i=1,\ldots,k} \in [L^{\infty}(\mathbb R /L\mathbb Z)]^{2k}$ 
satisfying:
\begin{align} \label{eq_initialyoung1}
&0 \leq \alpha_i^0(x) \leq 1 \quad \sum_{i=0}^k \alpha_i^0(x) = 1 \quad \text{a.e. in $\mathbb R/L\mathbb Z.$}\\
 \label{eq_initialyoung1bis}
&\dfrac 4M \leq \rho_i^0(x) \leq \dfrac M4  \quad \text{a.e. in $\mathbb R/L\mathbb Z.$}
\end{align}
and we consider the initial data for \eqref{eq_youngapp} that reads:
\begin{equation} \label{eq_initialyoung2}
\nu^0  = \sum_{i=1}^k \alpha_i^0(x) \delta_{\xi =\rho_i^0(x)}.
\end{equation}
We show that  we can construct a solution to \eqref{eq_youngapp} with the same structure (under the above assumptions on the velocity $u$ and $g).$ Namely, there holds:
\begin{lemma} \label{lem_youngcons}
Let \eqref{eq_initialyoung1}-\eqref{eq_initialyoung1bis}-\eqref{eq_initialyoung2} hold true. There exists $T_0<T$ and 
$$
(\alpha_i,\rho_i) \in L^{\infty}((0,T_0) \times (\mathbb R /L\mathbb Z) ) \cap C([0,T_0] ; L^1(\mathbb R /L\mathbb Z))
$$ satisfying
\begin{align}  \label{eq_cvx1}
& 0 \leq \alpha_i(t,x) \leq 1 \quad \sum_{i=0}^k \alpha_i(t,x) = 1 \quad \text{a.e.} \\
& \dfrac 2{M} \leq \rho_i(t,x) \leq \dfrac{M}2  \quad \text{a.e.}  \label{eq_cvx2}
\end{align}
such that $\nu = \sum_{i=1}^{k} \alpha_i \delta_{\xi=\rho_i} \in C([0,T_0] ; \mathcal Y_{\sharp})$ is a solution to \eqref{eq_youngapp}-\eqref{eq_youngappini}.
\end{lemma}
\begin{proof}
The proof is straightforward.   Let $\nu = \sum_{i=1}^{k} \alpha_i \delta_{\xi=\rho_i}$  with $(\alpha_i,\rho_i)$
as in the statement of the theorem. We have thus that,
for arbitrary $\phi \in C_c(\mathbb R \times [0,M]),$ there holds:
\begin{equation} \label{eq_formyoung}
\langle \nu_t , \phi \rangle = \int_{0}^L \sum_{i=1}^k \alpha_i(t,x) \phi(x,\rho_i(t,x)) {\rm d}x.
\end{equation}
Hence $\langle \nu , \phi \rangle \in C([0,T_0])$ with $|\langle \nu_t, \phi \rangle | \leq L \|\phi\|_{L^{\infty}_{\sharp}}$  so that
we have indeed $\nu \in C([0,T_0] ; \mathcal Y_{\sharp}).$   
 
Then, applying a classical density argument, we obtain that $\nu$ is charaterized by its action on tensorized test-functions $(x,\xi) \mapsto \psi(x) \beta(\xi).$ 
Plugging $\psi \otimes \beta$ as test-function  in \eqref{eq_youngapp}-\eqref{eq_youngappini}, we obtain the following equations:
\begin{eqnarray*}
\partial_t  \sum_{i=1}^k \alpha_i \beta(\rho_i)  + \partial_1  \sum_{i=1}^k \alpha_i \beta(\rho_i) u_1(\cdot,\cdot,\rho_i) 
& =& \sum_{i=1}^k \alpha_i u_{2}(\cdot,\cdot,\rho_i) \beta'(\rho_i) - \sum_{i=1}^k \alpha_i g(\cdot,\cdot,\rho_i) \beta(\rho_i).\\
  \sum_{i=1}^k \alpha_i \beta(\rho_i)_{|_{t=0}} &=& \sum_{i=1}^k \alpha_i^0 \beta(\rho_i^0).  
\end{eqnarray*}   
Finally, we obtain that $\nu$  is a solution to \eqref{eq_youngapp}-\eqref{eq_youngappini}  if the $(\alpha_i,\rho_i)$ satisfy simultaneously:
 \begin{align}
\partial_t \alpha_i + \partial_1 \left(\alpha_i u_{1}  \right) + \alpha_i g(\cdot,\cdot,\rho_i) & =0\\
[\alpha_i]_{|_{t=0}} &= \alpha_i^0
 \end{align}
 and
 \begin{align}
 \partial_t \rho_i + u_1 \partial_1 \rho_i - u_{2}(\cdot,\cdot,\rho_i) &=0 \\
 [\rho_i]_{|_{t=0}} &= \rho_i^0.
 \end{align}
Remark that we introduced that $u_1$ does not depend on $\rho_i$ (by assumption). Existence of a solution 
$$
(\alpha_i,\rho_i) \in L^{\infty}((0,T_0) \times (\mathbb R /L\mathbb Z) ) \cap C([0,T_0] ; L^1(\mathbb R /L\mathbb Z))
$$ 
to this system satisfying \eqref{eq_cvx1}-\eqref{eq_cvx2} follows from a straightforward adaptation of Di Perna-Lions arguments
in the spirit of  \cite[Lemma 2]{BrHi}.

\end{proof}

\section{Formal calculation versus Young measure method} 
 Let us compare in this appendix the system obtained through a formal WKB method and the
 system derived using kinetic formulation and characterization of the Young measures.
 With the Young measure method in the two-fluid setting, we get the following equation on  $\alpha_+$ : 
$$
\partial_t \alpha_+ + u \partial_x \alpha_+ + \alpha_+ \partial_x u = \dfrac{\alpha_+}{\mu_+} \left[ \dfrac{1}{\frac{\alpha_+}{\mu_+} + \frac{\alpha_-}{\mu_-}} \left( \partial_x u - \left(\alpha_+ \dfrac{p_+}{\mu_+} + \alpha_- \dfrac{p_-}{\mu_-}\right)  \right) + p_+ \right]
$$
Thus we get the following equation 
$$
\partial_t \alpha_+ + u \partial_x \alpha_+ = \kappa_1 \partial_x u + \kappa_2 
$$
with
\begin{eqnarray*}
\kappa_1 &=&  \dfrac{\alpha_+}{\mu_+}  \dfrac{1}{\frac{\alpha_+}{\mu_+} + \frac{\alpha_-}{\mu_-}} - \alpha_+  \\
		&=& \dfrac{\alpha_+}{\mu_+ } \dfrac{1 - \mu_+ (\frac{\alpha_+}{\mu_+} + \frac{\alpha_-}{\mu_-})}{\frac{\alpha_+}{\mu_+} + \frac{\alpha_-}{\mu_-}}\\
		&=& \dfrac{\alpha_+ \alpha_-}{\mu_+ } \dfrac{ 1 - \frac{\mu_+}{\mu_-}}{\frac{\alpha_+}{\mu_+} + \frac{\alpha_-}{\mu_-}} = \dfrac{\alpha_+ \alpha_- (\mu_- - \mu_+)}{\alpha_+ \mu_+ + \alpha_- \mu_-}.
\end{eqnarray*}
and 
\begin{eqnarray*}
\kappa_2 &=&   \dfrac{\alpha_+}{\mu_+} \left[ p_+ -\dfrac{1}{\frac{\alpha_+}{\mu_+} + \frac{\alpha_-}{\mu_-}}\left(\alpha_+ \dfrac{p_+}{\mu_+} + \alpha_- \dfrac{p_-}{\mu_-} \right) \right] \\
		&=& \dfrac{\alpha_+\alpha_-}{\mu_+\mu_-}\dfrac{1}{\frac{\alpha_+}{\mu_+} + \frac{\alpha_-}{\mu_-}} (p_+ - p_-) .\\
		&=& \dfrac{\alpha_+\alpha_-}{\alpha_-\mu_+ + \alpha_+ \mu_-}(p_+ - p_-) .
\end{eqnarray*}
This reads
$$
\partial_t \alpha_+ + u \partial_x \alpha_+ = \dfrac{\alpha_+\alpha_-}{\alpha_-\mu_+ + \alpha_+ \mu_-} \left[ (p_+ - p_-)  + (\mu_- - \mu_+) \partial_x u \right].
$$
As for the momentum equation, we obtain : 
$$
\partial_{t} (\rho u) + \partial_x (\rho u^2) -  \partial_x (m^{\infty} \partial_x u ) + \partial_x \pi^{\infty} = 0,
$$
where
$$
m^{\infty} =  \dfrac{1}{\frac{\alpha_+}{\mu_+} + \frac{\alpha_-}{\mu_-}}  = \dfrac{\mu_+ \mu_-}{\alpha_+ \mu_- + \alpha_- \mu_+}
$$
and $\pi^{\infty} = m^{\infty}\kappa^{\infty}$ with
$$
\kappa^{\infty}  = \alpha_+ \dfrac{p_+}{\mu_+} + \alpha_-\dfrac{p_-}{\mu_-} 
$$
and thus :
$$
\pi^{\infty} = m^{\infty}\kappa^{\infty}  =  \dfrac{\alpha_+ p_+  \mu_- + \alpha_- p_- \mu_+}{\alpha_+ \mu_- + \alpha_- \mu_+}.
$$
This is, up to the notations, the system obtained using the WKB method.

\bigskip

\noindent {\bf Acknowledgements.} The authors are partially supported by the ANR- 13-BS01-0003-01 project DYFICOLTI.

\end{document}